\documentclass[a4paper,12pt,leqno]{amsart}
\usepackage{latexsym}
\usepackage[all]{xy}

\usepackage{amssymb} 
\usepackage{amsmath} 
\usepackage{color}
\usepackage{comment}
\usepackage{mathtools}

\usepackage{graphicx}
\usepackage{stmaryrd}

\definecolor{gray}{gray}{0.7}
\definecolor{Gray}{gray}{0.3}

\usepackage{amscd}
\usepackage{mathdots}

\numberwithin{equation}{section}

\theoremstyle{break}
 \newtheorem{theorem}{Theorem}[section]
 \newtheorem{proposition}[theorem]{Proposition}
 \newtheorem{corollary}[theorem]{Corollary}
 \newtheorem{lemma}[theorem]{Lemma}

 \theoremstyle{definition}
 
 \newtheorem{remark}[theorem]{Remark}
 \newtheorem{example}[theorem]{Example}

\allowdisplaybreaks[3]

\def\C{\mathbb C}

\def\Z{\mathbb Z}
\def\gl{\mathfrak{gl}}
\def\b{\mathfrak{b}}
\def\ZZ{\mathcal{Z}}

\def\matX{\mathsf{x}}
\def\matN{\mathsf{n}}
\def\matS{\mathsf{s}}
\def\matA{\mathsf{a}}

\def\hh{\mathbf{h}}
\def\HessSp{\mathsf{H}}

\DeclareMathOperator{\Spec}{Spec}
\DeclareMathOperator{\GL}{GL}
\DeclareMathOperator{\U}{U}
\DeclareMathOperator{\Hess}{Hess}
\DeclareMathOperator{\Pet}{Pet}
\DeclareMathOperator{\Fl}{Fl}

\begin{document}

\title[Coordinate rings and cohomology rings of Hessenberg varieties]{Coordinate rings of regular semisimple Hessenberg varieties and cohomology rings of regular nilpotent Hessenberg varieties}
\author [T. Horiguchi]{Tatsuya Horiguchi}
\address{National Institute of Technology, Akashi College, 679-3, Nishioka, Uozumi-cho, Akashi, Hyogo 674-8501, Japan}
\email{tatsuya.horiguchi0103@gmail.com}

\subjclass[2020]{Primary 14M15, 05E05, 14N35}

\keywords{flag varieties, Hessenberg varieties, symmetric polynomials, quantization.} 

\begin{abstract}
The polynomials $f_{i,j}$ are introduced by Abe--Harada--Horiguchi--Masuda to produce an explicit presentation by generators and relations of the cohomology rings of regular nilpotent Hessenberg varieties. 
In this paper we quantize the polynomials $f_{i,j}$ by a method of Fomin--Gelfand--Postnikov. 
Our main result states that their quantizations $F_{i,j}$ are related to the coordinate rings of regular semisimple Hessenberg varieties.
This result yields a connection between the coordinate rings of regular semisimple Hessenberg varieties and the cohomology rings of regular nilpotent Hessenberg varieties. 
We also provide the quantized recursive formula for $F_{i,j}$. 
\end{abstract}

\maketitle

\setcounter{tocdepth}{1}

\tableofcontents

\section{Introduction}
\label{section:introduction}

Hessenberg varieties are closed subvarieties of flag varieties introduced in \cite{dMPS, dMS}.
The geometry and topology of Hessenberg varieties yield a fruitful intersection with other research areas such as quantum cohomology of flag varieties (\cite{Kos, Pet, Rie}), logarithmic derivation modules in hyperplene arrangements (\cite{AHMMS}), and Stanley's chromatic symmetric functions in graph theory (\cite{BroCho, ShaWac}).
These connections are discovered in regular nilpotent and regular semisimple Hessenberg varieties.
In this paper we connect the coordinate rings of regular semisimple Hessenberg varieties and the cohomology rings\footnote{Throughout this paper, all cohomology will be taken with complex coefficients unless otherwise specified.} of regular nilpotent Hessenberg varieties via quantizations.

Let $\Fl(\C^n)$ be the complex manifold of complete flags $V_\bullet =(V_1 \subset V_2 \subset \cdots \subset V_n = \C^n)$ in the $n$-dimensional linear space $\C^n$.
A Hessenberg function $\hh : \{1,2,\ldots,n\} \to \{1,2,\ldots,n\}$ is a weakly increasing function satisfying $\hh(j) \geq j$ for all $1 \leq j \leq n$. 
Let $\matX: \C^n \to \C^n$ be a linear operator and $\hh: \{1,2,\ldots,n\} \rightarrow \{1,2,\ldots,n\}$ a Hessenberg function. 
Then the Hessenberg variety $\Hess(\matX,\hh)$ associated to $\matX$ and $\hh$ is defined by
\begin{align*} 
\Hess(\matX,\hh) = \{ V_\bullet \in \Fl(\C^n) \mid \matX V_i \subset V_{\hh(i)} \ \textrm{for all} \ 1 \leq i \leq n \}. 
\end{align*}
Note that if we take $\hh(j)=n$ for all $j \in [n]$, then we have $\Hess(\matX,\hh)=\Fl(\C^n)$.
Throughout this paper we take $\matX$ as the regular nilpotent element $\matN$ and the regular semisimple element $\matS$ described in \eqref{eq:regular nilpotent semisimple}. 
Then $\Hess(\matN,\hh)$ is called regular nilpotent Hessenberg varieties and $\Hess(\matS,\hh)$ is called regular semisimple Hessenberg varieties. 
We first review their cohomology rings studied in \cite{AHHM}.

\subsection{Cohomology rings}
Borel represents in \cite{Bor53} the integral cohomology ring of $\Fl(\C^n)$ as a quotient of a polynomial ring $\Z[x_1,\ldots,x_n]$ by an ideal generated by the $i$-th elementary symmetric polynomials $e_i^{(n)}$ in the variables $x_1,\ldots,x_n$ for $i \in [n]$:
\begin{align*}
H^*(\Fl(\C^n);\Z) \cong \Z[x_1,\ldots,x_n]/(e_1^{(n)},\ldots,e_n^{(n)}).
\end{align*}
By \cite{AHHM} the restriction map $H^*(\Fl(\C^n)) \to H^*(\Hess(\matN,\hh))$ with complex coefficients is surjective and an explicit presentation for $H^*(\Hess(\matN,\hh))$ is given as follows.
For $i \geq j \geq 1$, the polynomials $f_{i,j}$ are defined by
\begin{align} \label{eq:Intro_fij}
f_{i,j} = \sum_{k=1}^{j} \left(\prod_{\ell=j+1}^i (x_k-x_\ell) \right) x_k
\end{align}
with the convention $\prod_{\ell=j+1}^j (x_k-x_\ell)=1$.
Then the cohomology ring of $\Hess(\matN,\hh)$ is 
\begin{align*}
H^*(\Hess(\matN,\hh)) \cong \C[x_1,\ldots,x_n]/(f_{\hh(1),1},\ldots,f_{\hh(n),n}).
\end{align*}
Note that the ideal $(f_{n,1},\ldots,f_{n,n})$ equals the ideal $(e_1^{(n)},\ldots,e_n^{(n)})$ in $\C[x_1,\ldots,x_n]$.
By using the recursive formula for $f_{i,j}$, the explicit presentation for $H^*(\Hess(\matN,\hh))$ above can be written as
\begin{align*} 
H^*(\Hess(\matN,\hh)) \cong \C[x_1,\ldots,x_n]/(f_{i,j} \mid i \geq \hh(j)). 
\end{align*}
This presentation include redundant generators, but it is useful to correlate to our main theorem described below.
These polynomials $f_{i,j}$ play an important role in this paper. 
Note that the polynomials $f_{i,j}$ are deeply related with the logarithmic derivation modules for ideal arrangements by \cite{AHMMS}, as mentioned in the beginning.

Further, in \cite{AHHM}, by using the explicit presentation for $H^*(\Hess(\matN,\hh))$, the cohomology ring $H^*(\Hess(\matN,\hh))$ is related with the cohomology ring $H^*(\Hess(\matS,\hh))$ through an action of the symmetric group $S_n$ on $H^*(\Hess(\matS,\hh))$ introduced by Tymoczko in \cite{Tym08}. 
Specifically, there is an isomorphism of graded $\C$-algebras
\begin{align*} 
H^*(\Hess(\matN,\hh)) \cong H^*(\Hess(\matS,\hh))^{S_n},
\end{align*}
where $H^*(\Hess(\matS,\hh))^{S_n}$ denotes the invariants in $H^*(\Hess(\matS,\hh))$ under the $S_n$-action. 
Note that the $S_n$-action on $H^*(\Hess(\matS,\hh))$ by Tymoczko is intimately connected with Stanley's chromatic symmetric functions for incomparability graphs of natural unit interval orders by \cite{BroCho, ShaWac}.

In this paper, we consider the quantizations of the polynomials $f_{i,j}$ by a similar way to construct quantum Schubert polynomials by Fomin--Gelfand--Postnikov in \cite{FGP}. 
The main theorem in this paper states that their quantizations of $f_{i,j}$ are related to the coordinate rings of regular semisimple Hessenberg varieties.
We breifly recount Fomin--Gelfand--Postnikov's quantization.

\subsection{Quantization}
The quantum analogue of Borel's presentation is obtained from Givental--Kim \cite{GK} and Ciocan-Fontanine \cite{Font95}. 
They prove that the quantum cohomology ring of $\Fl(\C^n)$ is isomorphic to the quotient of a polynomial ring $\Z[x_1,\ldots,x_n,q_1,\ldots,q_{n-1}]$ by an ideal generated by quantized elementary symmetric polynomials $E_1^{q \, (n)}, \ldots, E_n^{q \, (n)}$:
\begin{align*}
QH^*(\Fl(\C^n);\Z) \cong \Z[x_1,\ldots,x_n,q_1,\ldots,q_{n-1}]/(E_1^{q \, (n)}, \ldots, E_n^{q \, (n)}).
\end{align*}
Here, the quantized elementary symmetric polynomial $E_i^{q \, (n)}$ is defined by the coefficient of $\lambda^{n-i}$ multiplied by $(-1)^i$ for the characteristic polynomial in $\lambda$ of the matrix 
\begin{align*}
M_n^q = \left(
 \begin{array}{@{\,}ccccc@{\,}}
     x_1 & q_1 & 0 & \cdots & 0 \\
     -1 & x_2 & q_2 & \ddots & \vdots \\ 
      0 & \ddots & \ddots & \ddots & 0 \\ 
      \vdots & \ddots & -1 & x_{n-1} & q_{n-1} \\
      0 & \cdots & 0 & -1 & x_n 
 \end{array}
 \right).
\end{align*}
Namely, we have
\begin{align*}
\det(\lambda I_n - M_n^q) = \lambda^n - E_1^{q \, (n)} \lambda^{n-1} + E_2^{q \, (n)} \lambda^{n-2} + \cdots + (-1)^n E_n^{q \, (n)}
\end{align*}
where $I_n$ is the identity matrix of order $n$. 
Note that each $E_i^{q \, (n)}$ specializes to $e_i^{(n)}$ in setting $q_1=\cdots=q_{n-1}=0$. 
In \cite{FGP}, Fomin--Gelfand--Postnikov introduce difference operators on $\Z[x_1,\ldots,x_n,q_1,\ldots,q_{n-1}]$ and define the quantization map 
\begin{align*}
\psi: \Z[x_1,\ldots,x_n,q_1,\ldots,q_{n-1}] \to \Z[x_1,\ldots,x_n,q_1,\ldots,q_{n-1}]
\end{align*}
by using the difference operators.
Note that $\psi$ is an isomorphism as $\Z[q_1,\ldots,q_{n-1}]$-modules.
They prove that $\psi$ sends $e_i^{(k)}$ to $E_i^{q \, (k)}$ for any $1 \leq i \leq k \leq n$. 
The quantization map $\psi$ is \emph{not} a ring homomorphism.
Nevertheless, $\psi$ bijectively maps the ideal $(e_1^{(n)},\ldots,e_n^{(n)})$ onto the ideal $(E_1^{q \, (n)},\ldots,E_n^{q \, (n)})$. 
Therefore, the induced map is geometrically an isomorphism of $\Z[q_1,\ldots,q_{n-1}]$-modules:
\begin{align*}
H^*(\Fl(\C^n);\Z) \otimes \Z[q_1,\ldots,q_{n-1}] \to QH^*(\Fl(\C^n);\Z).
\end{align*}
In more general, we set 
\begin{align*}
e_{i_1,\ldots,i_m} = e_{i_1}^{(1)} \cdots e_{i_m}^{(m)} 
\end{align*}
for $i_1,\ldots,i_m$ with $0 \leq i_k \leq k$. 
These polynomials form an additive basis of the polynomial ring $\Z[x_1,x_2,\ldots]$ in infinitely many variables. 
In particular, the cohomology classes represented by the polynomials $e_{i_1,\ldots,i_{n-1}}$ form an additive basis for $H^*(\Fl(\C^n);\Z)$.
Similarly, let 
\begin{align*}
E_{i_1,\ldots,i_m}^q = E_{i_1}^{q \, (1)} \cdots E_{i_m}^{q \, (m)}. 
\end{align*}
Then, the cohomology classes represented by $E_{i_1,\ldots,i_{n-1}}^q$ form a $\Z[q_1,\ldots,q_{n-1}]$-module basis of $QH^*(\Fl(\C^n);\Z)$).
Fomin--Gelfand--Postnikov prove that $\psi$ maps $e_{i_1,\ldots,i_m}$ to $E_{i_1,\ldots,i_m}^q$ for any $1 \leq m \leq n$. 
From this point they introduce quantum Schubert polynomials as follows.
For $w \in S_n$, we write the Schubert polynomial $\mathfrak{S}_w$ as the unique expansion $\mathfrak{S}_w = \sum c_{i_1,\ldots,i_{n-1}} e_{i_1,\ldots,i_{n-1}}$ with integer coefficients $c_{i_1,\ldots,i_{n-1}}$. 
Then the quantum Schubert polynomial $\mathfrak{S}_w^q$ is defined by
\begin{align*}
\mathfrak{S}_w^q = \sum c_{i_1,\ldots,i_{n-1}} E_{i_1,\ldots,i_{n-1}}^q. 
\end{align*}
In particular, the quantum Schubert polynomial $\mathfrak{S}_w^q$ is the image of the Schubert polynomial $\mathfrak{S}_w$ under the quantization map $\psi$.
Based on this idea, we will define the quantizations of the polynomials $f_{i,j}$ in \eqref{eq:Intro_fij}.

\subsection{Coordinate rings}
We return to Hessenberg varieties.
If $\hh(j) = j+1$ for all $1 \leq j \leq n-1$ and $\hh(n) = n$, then the regular nilpotent Hessenberg variety $\Hess(\matN,\hh)$ is called the Peterson variety. 
We here denote the Peterson variety by $\Pet_n$. 
Peterson found a surprising connection between the coordinate ring for $\Pet_n$ and the quantum cohomology $QH^*(\Fl(\C^n))$ (\cite{Kos, Pet, Rie}).
We now explain this connection.
Let $U$ be the set of lower unipotent matrices: 
\begin{align*}  
U = \left\{ \left.
g=\left(
 \begin{array}{@{\,}ccccc@{\,}}
     1 &  &  &  &  \\
     x_{21} & 1 &  &  &  \\ 
     x_{31} & x_{32} & 1 &  &  \\ 
     \vdots& \vdots & \ddots & \ddots & \\
     x_{n1} & x_{n2} & \cdots & x_{n \, n-1} & 1 
 \end{array}
 \right) \right| x_{ij} \in \C \ (1 \leq j < i \leq n) \right\}. 
\end{align*}
Note that $U$ is also regarded as an open set around the identity element in $\Fl(\C^n)$. 
Peterson discovered that the coordinate ring $\C[\Pet_n \cap U]$ of $\Pet_n \cap U$ is isomorphic to the quantum cohomology ring of $\Fl(\C^n)$: 
\begin{align} \label{eq:Intro_coordinate_ring_Peterson} 
\C[\Pet_n \cap U] \cong QH^*(\Fl(\C^n)) \cong \C[x_1,\ldots,x_n,q_1,\ldots,q_{n-1}]/(E_1^{q \, (n)}, \ldots, E_n^{q \, (n)}).
\end{align}
Recently, the author and Shirato generalize in \cite{HorShi} this isomorphism to regular nilpotent Hessenberg varieties in algebraic sense as follows. 
Let $z_{ij}$ be the coordinate functions on $U$: 
\begin{align*}
z_{ij}(g) = x_{ij} \ \ \ \textrm{for} \ g \in U.
\end{align*} 
The flag represented by $g \in U$ belongs to $\Hess(\matX,\hh)$ if and only if $(g^{-1}\matX g)_{ij} =0$ for all $i > \hh(j)$. 
We set 
\begin{align*} 
\nu_{i,j}(g) = (g^{-1} \matN g)_{ij} \ \ \textrm{and} \ \ \xi_{i,j}(g) = (g^{-1} \matS g)_{ij} \ \ \ \textrm{for} \ g \in U.
\end{align*} 
Then the coordinate rings of $\Hess(\matN,\hh) \cap U$ and $\Hess(\matS,\hh) \cap U$ are given by 
\begin{align*} 
&\C[\Hess(\matN,\hh) \cap U] = \C[z_{ij} \mid 1 \leq j < i \leq n]/(\nu_{i,j} \mid i > \hh(j)); \\ 
&\C[\Hess(\matS,\hh) \cap U] = \C[z_{ij} \mid 1 \leq j < i \leq n]/(\xi_{i,j} \mid i > \hh(j)), 
\end{align*}
where $\Hess(\matN,\hh) \cap U$ and $\Hess(\matS,\hh) \cap U$ stand for the scheme-theoretic intersections in Introduction. 

Consider a generalization $M_n^{q_{rs}}$ (or simply denoted by $M_n$) of the matrix $M_n^q$ defined by \begin{align*}
M_n^{q_{rs}} = M_n = \left(
 \begin{array}{@{\,}ccccc@{\,}}
     x_1 & q_{12} & q_{13} & \cdots & q_{1n} \\
     -1 & x_2 & q_{23} & \cdots & q_{2n} \\ 
      0 & \ddots & \ddots & \ddots & \vdots \\ 
      \vdots & \ddots & -1 & x_{n-1} & q_{n-1 \, n} \\
      0 & \cdots & 0 & -1 & x_n 
 \end{array}
 \right).
\end{align*}
By a similar construction of quantized elementary symmetric polynomials $E_i^{q \, (n)}$, we define $E_i^{(n)}$ by
the coefficient of $\lambda^{n-i}$ multiplied by $(-1)^i$ for the characteristic polynomial in $\lambda$ of the matrix $M_n$ for each $1 \leq i \leq n$. 
We set 
\begin{align*} 
Q_n = \C[x_1,\ldots,x_n,q_{rs} \mid 1 \leq r < s \leq n]/(E_1^{(n)}, \ldots, E_n^{(n)})
\end{align*}
and define the map 
\begin{align*} 
\varphi: \C[z_{ij} \mid 1 \leq j < i \leq n] \to Q_n; \ \ \ z_{ij} \mapsto E_{i-j}^{(n-j)}.
\end{align*}
By \cite{HorShi}, the map $\varphi$ above is an isomorphism of $\C$-algebras and 
\begin{align*}
\varphi^{-1}(q_{rs}) = -\nu_{n+1-r, \, n+1-s} \ \ \ \textrm{for} \ 1 \leq r < s \leq n.
\end{align*}
This implies that for any Hessenberg function $\hh: \{1,2,\ldots,n\} \to \{1,2,\ldots,n\}$, we obtain an isomorphism of $\C$-algebras:
\begin{align} \label{eq:Intro_coordinate_ring_regular_nilpotent} 
\C[\Hess(\matN,\hh) \cap U] \cong \frac{\C[x_1,\ldots,x_n, q_{rs} \mid 2 \leq s \leq n, n-\hh(n+1-s)<r<s]}{({}^{\hh}E_1^{(n)}, \ldots, {}^{\hh}E_n^{(n)})}, 
\end{align} 
where ${}^{\hh}E_i^{(n)}$ is defined by 
\begin{align*}
{}^{\hh}E_i^{(n)} \coloneqq E_i^{(n)}|_{q_{rs}=0 \ (2 \leq s \leq n \ \textrm{and} \ 1 \leq r \leq n-\hh(n+1-s))}. 
\end{align*}
In particular, if $\hh(j) = j+1$ for all $1 \leq j \leq n-1$ and $\hh(n) = n$, then the isomorphism \eqref{eq:Intro_coordinate_ring_regular_nilpotent} leads us to the isomorphism \eqref{eq:Intro_coordinate_ring_Peterson} in algebraic sense. 

\subsection{Main result}
The time is ripe to state our main theorem. 
We quantize the polynomial $f_{i,j}$ defined in \eqref{eq:Intro_fij} by a combinatorial method of Fomin--Gelfand--Postnikov's quantization.
We write the unique expansion $f_{i,j} = \sum c_{i_1,\ldots,i_m} e_{i_1,\ldots,i_m}$ with integer coefficients $c_{i_1,\ldots,i_m}$. 
Then its quantization $F_{i,j}$ is defined by
\begin{align*}
F_{i,j} = \sum c_{i_1,\ldots,i_m} E_{i_1,\ldots,i_m}
\end{align*}
with the notation $E_{i_1,\ldots,i_m} = E_{i_1}^{(1)} \cdots E_{i_m}^{(m)}$.
When we set $q_{rs}=0$ for all $r < s$, the polynomial $F_{i,j}$ specializes to $f_{i,j}$ since $E_i^{(k)}$ specializes to $e_i^{(k)}$. 

\begin{theorem} \label{theorem:Intro_main}
\begin{enumerate}
\item The polynomial $F_{i,j}$ belongs to $\C[x_1,\ldots,x_i, q_{rs} \mid 1 \leq r < s \leq i]$ for all $i \geq j \geq 1$. 
\item The ideal generated by $E_1^{(n)}, \ldots, E_n^{(n)}$ equals the ideal generated by $F_{n,1}, \ldots, F_{n,n}$ in the polynomial ring $\C[x_1,\ldots,x_n,q_{rs} \mid 1 \leq r < s \leq n]$. 
In particular, we have 
\begin{align*}
Q_n = \C[x_1,\ldots,x_n,q_{rs} \mid 1 \leq r < s \leq n]/(F_{n,1}, \ldots, F_{n,n}).
\end{align*}
\item For any $1 \leq j \leq i \leq n-1$, we have 
\begin{align*}
\varphi^{-1}(F_{i,j}) = (-1)^{i-j} \xi_{n+1-j, \, n-i}.
\end{align*}
In particular, for any Hessenberg function $\hh: \{1,2,\ldots,n\} \to \{1,2,\ldots,n\}$, the isomorphism $\varphi$ induces an isomorphism of $\C$-algebras:
\begin{align} \label{eq:Intro_coordinate_ring_regular_semisimple}
\C[\Hess(\matS,\hh) \cap U] \cong \C[x_1,\ldots,x_n,q_{rs} \mid 1 \leq r < s \leq n]/(F_{i,j} \mid i \geq \hh^*(j)). 
\end{align}
Here, $\hh^*$ denotes the dual of $\hh$. $($See \eqref{eq:HessenbergFunctionDual} for the definition.$)$ 
\end{enumerate}
\end{theorem}

Theorem~\ref{theorem:Intro_main} gives rise to an unexpected relation of the coordinate ring $\C[\Hess(\matS,\hh) \cap U]$ to the cohomology ring $H^*(\Hess(\matN,\hh))$.

\begin{corollary} \label{corollary:Intro_main}
Let $\hh: \{1,2,\ldots,n\} \to \{1,2,\ldots,n\}$ be a Hessenberg function. 
Then there is an isomorphism of $\C$-algebras:
\begin{align*} 
\C[\Hess(\matS,\hh) \cap U]/(\nu_{i,j} \mid i>j) \cong H^*(\Hess(\matN,\hh)) \cong H^*(\Hess(\matS,\hh))^{S_n}.
\end{align*}
\end{corollary}

The results \eqref{eq:Intro_coordinate_ring_regular_nilpotent}, \eqref{eq:Intro_coordinate_ring_regular_semisimple}, and Corollary~\ref{corollary:Intro_main} are summarized in the diagram below.
\begin{align*}
\xymatrix{
& Q_n \ar[ld]_{q_{rs}=0 \atop \hspace{-70pt} (2 \leq s \leq n, \, 1 \leq r \leq n-\hh(n+1-s))} \ar@/_38pt/[rdd]_{F_{i,j}=0 \ (1 \leq i \leq n-1, \, i \geq \hh^*(j)) \atop q_{rs}=0 \ (1 \leq r < s \leq n)} \ar[rd]^{\hspace{-30pt} F_{i,j}=0 \atop \hspace{15pt} (1 \leq i \leq n-1, \, i \geq \hh^*(j))} & \\
\C[\Hess(\matN,\hh) \cap U] & & \C[\Hess(\matS,\hh) \cap U] \ar[d]^{\nu_{i,j}=0 \ (1 \leq j < i \leq n)} \\
& & H^*(\Hess(\matN,\hh)) 
}
\end{align*}

The polynomials $f_{i,j}$ are originally defined by a recursive formula in \cite{AHHM}. (See \eqref{eq:recursive_formula_fij}.)
We also give the recursive formula for $F_{i,j}$.

\begin{theorem}[Quantized recursive formula for $F_{i,j}$] \label{theorem:Intro_quantized recursive formula for Fij}
The polynomials $F_{i,j} \ (i \geq j \geq 1)$ satisfy the following recursive formula:
\begin{align*} 
\begin{split}
F_{j,j} &= x_1+ \cdots + x_j \ \ \ \textrm{$($the base case$)$;} \\
F_{i,j} &= F_{i-1,j-1} + \left( x_j \, F_{i-1,j} + \sum_{k=1}^{i-j-1} (-1)^k q_{j \, j+k} \, F_{i-1,j+k} \right) \\
& \hspace{25pt} - \left( x_i \, F_{i-1,j} + \sum_{k=1}^{i-j-1} (-1)^k q_{i-k \, i} \, F_{i-k-1,j} \right) +(-1)^{i-j}(i-j+1) q_{ji} \ \ \ \textrm{for} \ i>j.
\end{split}
\end{align*}
Here, we take the convention $F_{*,0}=0$ for arbitrary $*$, and $\sum_{k=1}^{i-j-1} (-1)^k q_{j \, j+k} \, F_{i-1,j+k} = \sum_{k=1}^{i-j-1} (-1)^k q_{i-k \, i} \, F_{i-k-1,j} =0$ whenever $i=j+1$.
\end{theorem}

In the classical limit $q_{rs}= 0$ for all $r <s$, the quantized recursive formula for $F_{i,j}$ becomes the recursive formula for $f_{i,j}$.

We present below the outline of the paper.
After reviewing the definitions and basic properties for Hessenberg varieties and Hessenberg schemes in Section~\ref{section:Hessenberg varieties and schemes}, we explain the necessary background of the cohomology rings of regular nilpotent Hessenberg varieties $\Hess(\matN,\hh)$ and the coordinate rings of $\Hess(\matN,\hh) \cap U$ in Section~\ref{sect:cohomology rings of regular nilpotent Hessenberg varieties} and 
Section~\ref{sect:quantization and coordinate rings}, respectively. 
In Section~\ref{sect:the main theorem}, we state the main theorem in this paper.
We study the polynomials $f_{i,j}$ and their quantizations $F_{i,j}$ in Section~\ref{sect:deterinant formula for fij}. 
We particularly derive a determinant formula for $f_{i,j}$ that yields the expansion of $f_{i,j}$ in the basis $e_{i_1,\ldots,i_m}$. 
This allows us to derive a determinant formula for $F_{i,j}$. 
Section~\ref{sect:proof of main theorem} is devoted to the proof of the main theorem.
We finally provide the quantized recursive formula for $F_{i,j}$ in Section~\ref{sect:quantized recursive formula for Fij}.

\bigskip
\noindent \textbf{Acknowledgements.} 
The author is supported in part by JSPS KAKENHI Grant-in-Aid for Early-Career Scientists: 23K12981.

\section{Hessenberg varieties and Hessenberg schemes}
\label{section:Hessenberg varieties and schemes}

\subsection{Hessenberg varieties}
We begin with the definition of flag varieties in type $A$. 
Throughout this paper, we frequently write
\begin{align*}
[n] \coloneqq \{1,2,\ldots,n\}
\end{align*}
for a positive integer $n$.
The flag variety $\Fl(\C^n)$ in type $A_{n-1}$ is the set of nested complex linear subspaces $V_\bullet \coloneqq (V_1 \subset V_2 \subset \dots \subset V_n = \C^n)$ of $\C^n$  where each $V_i$ is an $i$-dimensional complex subspace of $\C^n$.
A function $\hh: [n] \rightarrow [n]$ is a \emph{Hessenberg function} if it satisfies the following two conditions 
\begin{align*}
&\hh(j) \geq j \ \textrm{for all} \ j \in [n]; \\
&\hh(1) \leq \hh(2) \leq \dots \leq \hh(n).
\end{align*} 
Note that $\hh(n)=n$ by definition.
We denote a Hessenberg function $\hh$ by listing its values in sequence as follows
\begin{align*}
\hh=(\hh(1), \hh(2), \ldots, \hh(n)=n).
\end{align*}
For a linear operator $\matX: \C^n \to \C^n$ and a Hessenberg function $\hh: [n] \rightarrow [n]$, the associated \emph{Hessenberg variety} $\Hess(\matX,\hh)$ is defined as follows
\begin{align*} 
\Hess(\matX,\hh) \coloneqq \{ V_\bullet \in \Fl(\C^n) \mid \matX V_i \subset V_{\hh(i)} \ \textrm{for all} \ i \in [n] \}. 
\end{align*}
Note that if $\hh=(n,n,\ldots,n)$, then $\Hess(\matX,\hh)=\Fl(\C^n)$.

Let $\gl_n(\C)$ be the set of $n \times n$ matrices and $\b$ the set of upper triangular matrices in $\gl_n(\C)$.
A subspace $\HessSp \subset \gl_n(\C)$ is a \emph{Hessenberg space} if $\HessSp$ contains $\b$ and 
$\HessSp$ is stable under the adjoint action of $\b$.
To each Hessenberg function $\hh: [n] \to [n]$, we can assign a Hessenberg space by
\begin{align} \label{eq:Hessenberg space}
\HessSp = \{(a_{ij})_{i,j \in [n]} \in \gl_n(\C) \mid a_{ij} = 0 \ \textrm{if} \ i > \hh(j) \}.
\end{align}
We call $\HessSp$ in \eqref{eq:Hessenberg space} the \emph{Hessenberg space associated with} $\hh$. 
One can verify that this correspondence gives one-to-one correspondence between the set of Hessenberg functions and the set of Hessenberg spaces.
For this reason, a Hessenberg function $\hh$ is pictorially regarded as shaded boxes in a configuration of boxes on a square grid of size $n \times n$, where the shaded boxes consist of boxes in the $i$-th row and the $j$-th column with $i \leq \hh(j)$ for $(i, j) \in [n] \times [n]$.

\begin{example} \label{example:Hess_func_(3,4,4,5,5)}
Consider $n=5$ and a Hessenberg function $\hh=(3,4,4,5,5)$. 
The corresponding Hessenberg space $\HessSp$ is given by 
\begin{align*}
\HessSp=\left\{\begin{pmatrix}
a_{11} & a_{12} & a_{13} & a_{14} & a_{15} \\
a_{21} & a_{22} & a_{23} & a_{24} & a_{25} \\
a_{31} & a_{32} & a_{33} & a_{34} & a_{35} \\
0 & a_{42} & a_{43} & a_{44} & a_{45} \\
0 & 0 & 0 & a_{54} & a_{55} \\
\end{pmatrix} \middle| a_{ij} \in \C \ \textrm{for} \ i \leq \hh(j) \right\}.
\end{align*}
Also, the configuration of the shaded boxes for $\hh$ is shown in Figure~\ref{picture:Hessenberg_function}.
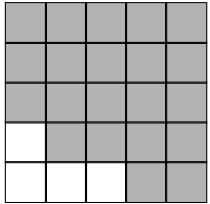
\begin{figure}[h]
\begin{center}
\begin{picture}(75,75)
\put(0,63){\colorbox{gray}}
\put(0,67){\colorbox{gray}}
\put(0,72){\colorbox{gray}}
\put(4,63){\colorbox{gray}}
\put(4,67){\colorbox{gray}}
\put(4,72){\colorbox{gray}}
\put(9,63){\colorbox{gray}}
\put(9,67){\colorbox{gray}}
\put(9,72){\colorbox{gray}}

\put(15,63){\colorbox{gray}}
\put(15,67){\colorbox{gray}}
\put(15,72){\colorbox{gray}}
\put(19,63){\colorbox{gray}}
\put(19,67){\colorbox{gray}}
\put(19,72){\colorbox{gray}}
\put(24,63){\colorbox{gray}}
\put(24,67){\colorbox{gray}}
\put(24,72){\colorbox{gray}}

\put(30,63){\colorbox{gray}}
\put(30,67){\colorbox{gray}}
\put(30,72){\colorbox{gray}}
\put(34,63){\colorbox{gray}}
\put(34,67){\colorbox{gray}}
\put(34,72){\colorbox{gray}}
\put(39,63){\colorbox{gray}}
\put(39,67){\colorbox{gray}}
\put(39,72){\colorbox{gray}}

\put(45,63){\colorbox{gray}}
\put(45,67){\colorbox{gray}}
\put(45,72){\colorbox{gray}}
\put(49,63){\colorbox{gray}}
\put(49,67){\colorbox{gray}}
\put(49,72){\colorbox{gray}}
\put(54,63){\colorbox{gray}}
\put(54,67){\colorbox{gray}}
\put(54,72){\colorbox{gray}}

\put(60,63){\colorbox{gray}}
\put(60,67){\colorbox{gray}}
\put(60,72){\colorbox{gray}}
\put(64,63){\colorbox{gray}}
\put(64,67){\colorbox{gray}}
\put(64,72){\colorbox{gray}}
\put(69,63){\colorbox{gray}}
\put(69,67){\colorbox{gray}}
\put(69,72){\colorbox{gray}}

\put(0,48){\colorbox{gray}}
\put(0,52){\colorbox{gray}}
\put(0,57){\colorbox{gray}}
\put(4,48){\colorbox{gray}}
\put(4,52){\colorbox{gray}}
\put(4,57){\colorbox{gray}}
\put(9,48){\colorbox{gray}}
\put(9,52){\colorbox{gray}}
\put(9,57){\colorbox{gray}}

\put(15,48){\colorbox{gray}}
\put(15,52){\colorbox{gray}}
\put(15,57){\colorbox{gray}}
\put(19,48){\colorbox{gray}}
\put(19,52){\colorbox{gray}}
\put(19,57){\colorbox{gray}}
\put(24,48){\colorbox{gray}}
\put(24,52){\colorbox{gray}}
\put(24,57){\colorbox{gray}}

\put(30,48){\colorbox{gray}}
\put(30,52){\colorbox{gray}}
\put(30,57){\colorbox{gray}}
\put(34,48){\colorbox{gray}}
\put(34,52){\colorbox{gray}}
\put(34,57){\colorbox{gray}}
\put(39,48){\colorbox{gray}}
\put(39,52){\colorbox{gray}}
\put(39,57){\colorbox{gray}}

\put(45,48){\colorbox{gray}}
\put(45,52){\colorbox{gray}}
\put(45,57){\colorbox{gray}}
\put(49,48){\colorbox{gray}}
\put(49,52){\colorbox{gray}}
\put(49,57){\colorbox{gray}}
\put(54,48){\colorbox{gray}}
\put(54,52){\colorbox{gray}}
\put(54,57){\colorbox{gray}}

\put(60,48){\colorbox{gray}}
\put(60,52){\colorbox{gray}}
\put(60,57){\colorbox{gray}}
\put(64,48){\colorbox{gray}}
\put(64,52){\colorbox{gray}}
\put(64,57){\colorbox{gray}}
\put(69,48){\colorbox{gray}}
\put(69,52){\colorbox{gray}}
\put(69,57){\colorbox{gray}}

\put(0,33){\colorbox{gray}}
\put(0,37){\colorbox{gray}}
\put(0,42){\colorbox{gray}}
\put(4,33){\colorbox{gray}}
\put(4,37){\colorbox{gray}}
\put(4,42){\colorbox{gray}}
\put(9,33){\colorbox{gray}}
\put(9,37){\colorbox{gray}}
\put(9,42){\colorbox{gray}}

\put(15,33){\colorbox{gray}}
\put(15,37){\colorbox{gray}}
\put(15,42){\colorbox{gray}}
\put(19,33){\colorbox{gray}}
\put(19,37){\colorbox{gray}}
\put(19,42){\colorbox{gray}}
\put(24,33){\colorbox{gray}}
\put(24,37){\colorbox{gray}}
\put(24,42){\colorbox{gray}}

\put(30,33){\colorbox{gray}}
\put(30,37){\colorbox{gray}}
\put(30,42){\colorbox{gray}}
\put(34,33){\colorbox{gray}}
\put(34,37){\colorbox{gray}}
\put(34,42){\colorbox{gray}}
\put(39,33){\colorbox{gray}}
\put(39,37){\colorbox{gray}}
\put(39,42){\colorbox{gray}}

\put(45,33){\colorbox{gray}}
\put(45,37){\colorbox{gray}}
\put(45,42){\colorbox{gray}}
\put(49,33){\colorbox{gray}}
\put(49,37){\colorbox{gray}}
\put(49,42){\colorbox{gray}}
\put(54,33){\colorbox{gray}}
\put(54,37){\colorbox{gray}}
\put(54,42){\colorbox{gray}}

\put(60,33){\colorbox{gray}}
\put(60,37){\colorbox{gray}}
\put(60,42){\colorbox{gray}}
\put(64,33){\colorbox{gray}}
\put(64,37){\colorbox{gray}}
\put(64,42){\colorbox{gray}}
\put(69,33){\colorbox{gray}}
\put(69,37){\colorbox{gray}}
\put(69,42){\colorbox{gray}}


\put(15,18){\colorbox{gray}}
\put(15,22){\colorbox{gray}}
\put(15,27){\colorbox{gray}}
\put(19,18){\colorbox{gray}}
\put(19,22){\colorbox{gray}}
\put(19,27){\colorbox{gray}}
\put(24,18){\colorbox{gray}}
\put(24,22){\colorbox{gray}}
\put(24,27){\colorbox{gray}}

\put(30,18){\colorbox{gray}}
\put(30,22){\colorbox{gray}}
\put(30,27){\colorbox{gray}}
\put(34,18){\colorbox{gray}}
\put(34,22){\colorbox{gray}}
\put(34,27){\colorbox{gray}}
\put(39,18){\colorbox{gray}}
\put(39,22){\colorbox{gray}}
\put(39,27){\colorbox{gray}}

\put(45,18){\colorbox{gray}}
\put(45,22){\colorbox{gray}}
\put(45,27){\colorbox{gray}}
\put(49,18){\colorbox{gray}}
\put(49,22){\colorbox{gray}}
\put(49,27){\colorbox{gray}}
\put(54,18){\colorbox{gray}}
\put(54,22){\colorbox{gray}}
\put(54,27){\colorbox{gray}}

\put(60,18){\colorbox{gray}}
\put(60,22){\colorbox{gray}}
\put(60,27){\colorbox{gray}}
\put(64,18){\colorbox{gray}}
\put(64,22){\colorbox{gray}}
\put(64,27){\colorbox{gray}}
\put(69,18){\colorbox{gray}}
\put(69,22){\colorbox{gray}}
\put(69,27){\colorbox{gray}}




\put(45,3){\colorbox{gray}}
\put(45,7){\colorbox{gray}}
\put(45,12){\colorbox{gray}}
\put(49,3){\colorbox{gray}}
\put(49,7){\colorbox{gray}}
\put(49,12){\colorbox{gray}}
\put(54,3){\colorbox{gray}}
\put(54,7){\colorbox{gray}}
\put(54,12){\colorbox{gray}}

\put(60,3){\colorbox{gray}}
\put(60,7){\colorbox{gray}}
\put(60,12){\colorbox{gray}}
\put(64,3){\colorbox{gray}}
\put(64,7){\colorbox{gray}}
\put(64,12){\colorbox{gray}}
\put(69,3){\colorbox{gray}}
\put(69,7){\colorbox{gray}}
\put(69,12){\colorbox{gray}}

\put(0,0){\framebox(15,15)}
\put(15,0){\framebox(15,15)}
\put(30,0){\framebox(15,15)}
\put(45,0){\framebox(15,15)}
\put(60,0){\framebox(15,15)}
\put(0,15){\framebox(15,15)}
\put(15,15){\framebox(15,15)}
\put(30,15){\framebox(15,15)}
\put(45,15){\framebox(15,15)}
\put(60,15){\framebox(15,15)}
\put(0,30){\framebox(15,15)}
\put(15,30){\framebox(15,15)}
\put(30,30){\framebox(15,15)}
\put(45,30){\framebox(15,15)}
\put(60,30){\framebox(15,15)}
\put(0,45){\framebox(15,15)}
\put(15,45){\framebox(15,15)}
\put(30,45){\framebox(15,15)}
\put(45,45){\framebox(15,15)}
\put(60,45){\framebox(15,15)}
\put(0,60){\framebox(15,15)}
\put(15,60){\framebox(15,15)}
\put(30,60){\framebox(15,15)}
\put(45,60){\framebox(15,15)}
\put(60,60){\framebox(15,15)}
\end{picture}
\end{center}
\vspace{-10pt}
\caption{The configuration corresponding to $\hh=(3,4,4,5,5)$.}
\label{picture:Hessenberg_function}
\end{figure}
\end{example}

Let $B$ be the set of upper triangular matrices in the general linear group $\GL_n(\C)$. 
As is well-known, the flag variety $\Fl(\C^n)$ is identified with $\GL_n(\C)/B$ by sending $gB \in \GL_n(\C)/B$ to the flag $V_\bullet \in \Fl(\C^n)$ such that $V_j$ is generated by the first $j$ column vectors of $g$ for any $j \in [n]$.
Under this identification, the Hessenberg variety $\Hess(\matX,\hh)$ is rewritten as follows
\begin{align} \label{eq:Hessenberg_variety}
\Hess(\matX,\HessSp) = \{ gB \in \GL_n(\C)/B \mid g^{-1} \matX g \in \HessSp \},
\end{align}
where $\HessSp$ is the Hessenberg space associated with $\hh$. 
Note that for any $p \in \GL_n(\C)$, the assignment $gB \mapsto p^{-1}gB$ gives an isomorphism $\Hess(\matX,\HessSp) \cong \Hess(p^{-1}\matX p,\HessSp)$. 
Thus, we may assume that $\matX$ in $\Hess(\matX,\HessSp)$ is in Jordan canonical form. 
Recalling that the flag variety $\GL_n(\C)/B$ is covered by affine coordinate patches, each isomorphic to $\C^{\frac{1}{2}n(n-1)}$, the defining equations of the Hessenberg variety $\Hess(\matX,\HessSp)$ is locally given by $(g^{-1} \matX g)_{ij} = 0$ for all $i>\hh(j)$ on each affine coordinate chart.
To be more precise, we set 
\begin{align} \label{eq:open set U} 
U = \left\{ \left.
g=\left(
 \begin{array}{@{\,}ccccc@{\,}}
     1 &  &  &  &  \\
     x_{21} & 1 &  &  &  \\ 
     x_{31} & x_{32} & 1 &  &  \\ 
     \vdots& \vdots & \ddots & \ddots & \\
     x_{n1} & x_{n2} & \cdots & x_{n \, n-1} & 1 
 \end{array}
 \right) \right| x_{ij} \in \C \ (1 \leq j < i \leq n) \right\} \cong \C^{\frac{1}{2}n(n-1)}. 
\end{align}
Then the map $U \to \GL_n(\C)/B; g \mapsto gB$ is an open embedding. 
By abuse of notation, we also denote by $U$ its image in $\GL_n(\C)/B$.
Let $S_n$ be the symmetric group on $[n]$.
Then the set of translates $wU$ of $U$ by $w \in S_n$ forms an open cover of $\GL_n(\C)/B$ where the permutation $w \in S_n$ is identified with the permutation matrix $w=(\delta_{i,w(j)})_{i,j}$. 
Note that $\delta_{i,j}$ denotes the Kronecker delta.
By the definition \eqref{eq:Hessenberg_variety},  the defining equations of the Hessenberg variety $\Hess(\matX,\HessSp)$ on $wU$ is given by $(g^{-1} \matX g)_{ij} = 0$ for all $i>h(j)$ where $g \in wU$. 

\subsection{Hessenberg schemes}
We can reconsider Hessenberg varieties as schemes. 
We now briefly recount. 
Let $\HessSp$ be a Hessenberg space.
Since $\HessSp$ is stable under the adjoint action of $B$, this induces a $B$-action on the quotient space $\gl_n(\C)/\HessSp$. 
Let $\overline{\matX}$ be the image of $\matX \in \gl_n(\C)$ under the surjection $\gl_n(\C) \rightarrow \gl_n(\C)/\HessSp$. 
Fix $\matX \in \gl_n(\C)$ and we define a section $s_\matX$ of the vector bundle $\GL_n(\C) \times_B (\gl_n(\C)/\HessSp)$ over the flag variety $\GL_n(\C)/B$ by
\begin{align*}
s_\matX: \GL_n(\C)/B \rightarrow \GL_n(\C) \times_B (\gl_n(\C)/\HessSp); \ gB \mapsto [g,\overline{g^{-1} \matX g}],
\end{align*}
where $[g,\overline{\matX}]$ is the image of $(g,\overline{\matX}) \in \GL_n(\C) \times (\gl_n(\C)/\HessSp)$ under the surjection $\GL_n(\C) \times (\gl_n(\C)/\HessSp) \rightarrow \GL_n(\C) \times_B (\gl_n(\C)/\HessSp)$ such that $[g,\overline{\matX}]=[gb,\overline{b^{-1}\matX b}]$ for all $b \in B$. 
By the definition \eqref{eq:Hessenberg_variety} the zero set of $s_\matX$ is the Hessenberg variety $\Hess(\matX,\HessSp)$. 
From this point the \emph{Hessenberg scheme} $\ZZ(\matX,\HessSp)$ associated to $\matX \in \gl_n(\C)$ and a Hessenberg space $\HessSp \subset \gl_n(\C)$ is defined to be the zero scheme of $s_\matX$. 
Namely, if we write $\zeta_{i,j}^w (g) = (g^{-1} \matX g)_{ij}$ for $g \in wU$, then $\ZZ(\matX,\HessSp)$ is defined in $wU$ by the ideal generated by $\zeta_{i,j}^w$ for all $i > \hh(j)$ where $\hh$ is the Hessenberg function corresponding to $\HessSp$. 
Note that we also denote $\ZZ(\matX,\HessSp)$ by $\ZZ(\matX,\hh)$.

\begin{remark}
Hessenberg schemes are originally introduced by Anderson--Tymoczko \cite{AndTym} in terms of degeneracy loci in a flag bundle.
The definition of Hessenberg schemes explained above is introduced in \cite{ADGH}. 
Note that Hessenberg schemes introduced in \cite{ADGH} and \cite{AndTym} are the same.
\end{remark}

\subsection{Regular nilpotent and regular semisimple}
Throughout this paper we take a linear operator $\matX: \C^n \to \C^n$ as the following two matrices
\begin{equation} \label{eq:regular nilpotent semisimple} 
\matN = 
\begin{pmatrix}
0 & 1 & &   \\
     & \ddots & \ddots & \\
     &   & 0 & 1 \\
     &  & & 0 \\ 
\end{pmatrix}, \quad 
\matS = 
\begin{pmatrix}
1 &  & &   \\
     & 2 & & \\
     &   & \ddots &  \\
     &  & & n \\ 
\end{pmatrix}. 
\end{equation}
Then $\Hess(\matN,\HessSp)$ is called the \emph{regular nilpotent Hessenberg variety} and $\Hess(\matS,\HessSp)$ is called the \emph{regular semisimple Hessenberg variety}.
Similarly, we call $\ZZ(\matN,\HessSp)$ the \emph{regular nilpotent Hessenberg scheme} and $\ZZ(\matS,\HessSp)$ the \emph{regular semisimple Hessenberg scheme}, respectively.
We record their geometric properties. 

\begin{proposition} $($ \cite[Theorem~1.3]{AbeInsko}, \cite[Lemma~7.1]{AndTym}, \cite[Theorem~10.2]{SomTym}$)$ 
Let $\hh: [n] \to [n]$ be a Hessenberg function. 
Then the followings hold.
\begin{enumerate}
\item[(i)] The complex dimension of the regular nilpotent Hessenberg variety $\Hess(\matN,\hh)$ is equal to $\sum_{j=1}^n (\hh(j)-j)$. 
\item[(ii)] Every regular nilpotent Hessenberg variety $\Hess(\matN,\hh)$ is irreducible.
\item[(iii)] Suppose that $\hh(j) >j$ for any $j \in [n-1]$. 
Then the regular nilpotent Hessenberg variety $\Hess(\matN,\hh)$ is normal if and only if $\hh$ satisfies the condition that $\hh(i-1) > i$ or $\hh(i) > i + 1$ for all $1 < i < n-1$.  
\end{enumerate}
\end{proposition}

\begin{proposition} $($\cite[Proposition~A.1]{AndTym}, \cite[Theorem~6]{dMPS} $)$
Let $\hh: [n] \to [n]$ be a Hessenberg function.
Then the followings hold.
\begin{enumerate}
\item[(i)] The regular semisimple Hessenberg variety $\Hess(\matS,\hh)$ is equidimensional of complex dimension equal to $\sum_{j=1}^n (\hh(j)-j)$. 
\item[(ii)] The regular semisimple Hessenberg variety $\Hess(\matS,\hh)$ is connected and irreducible if and only if $\hh(j) >j$ for all $j \in [n-1]$. 
\item[(iii)] Every regular semisimple Hessenberg variety $\Hess(\matS,\hh)$ is smooth.
\end{enumerate}
\end{proposition}

\begin{proposition} $($\cite[Proposition~3.6]{ADGH}, \cite[Remark~3.5]{AndTym}$)$ \label{proposition:reduced}
Let $\hh: [n] \to [n]$ be a Hessenberg function.
Then the followings hold.
\begin{enumerate}
\item[(1)] If $\hh(j) >j$ for any $j \in [n-1]$, then the regular nilpotent Hessenberg scheme $\ZZ(\matN,\hh)$ is reduced. 
\item[(2)] Every regular semisimple Hessenberg scheme $\ZZ(\matS,\hh)$ is reduced.
\end{enumerate}
\end{proposition}

\section{Cohomology rings of regular nilpotent Hessenberg varieties} \label{sect:cohomology rings of regular nilpotent Hessenberg varieties}

We define polynomials $f_{i,j}$ for $i \geq j \geq 1$ by the following recursive formula:
\begin{align} \label{eq:recursive_formula_fij}
\begin{split}
f_{j,j} &= x_1+ \cdots + x_j \ \ \ \textrm{(the base case);} \\
f_{i,j} &= f_{i-1,j-1} + (x_j-x_i)f_{i-1,j} \ \ \ \textrm{for} \ i>j.
\end{split}
\end{align}
Here, we take the convention $f_{*,0}=0$ for arbitrary $*$.
Explicitly, we obtain from \cite[Lemma~6.5]{AHHM} that 
\begin{align} \label{eq:fij}
f_{i,j} \coloneqq \sum_{k=1}^{j} \left(\prod_{\ell=j+1}^i (x_k-x_\ell) \right) x_k
\end{align}
with the convention $\prod_{\ell=j+1}^i (x_k-x_\ell)=1$ whenever $i=j$.
Here, we define 
\begin{align} \label{eq:degree xi}
\deg x_i =2
\end{align} 
for all $i \in [n]$. 
Note that $f_{i,j}$ is homogeneous of degree $2(i-j+1)$. 

\begin{theorem} $($\cite[Theorem~A]{AHHM}$)$
Let $\hh: [n] \rightarrow [n]$ be a Hessenberg function.
Then there is an isomorphism of graded $\C$-algebras
\begin{align} \label{eq:cohomology_Hess(N,h)}
H^*(\Hess(\matN,\hh)) \cong \C[x_1,\ldots,x_n]/(f_{h(1),1},f_{h(2),2},\ldots,f_{h(n),n}).
\end{align}
\end{theorem}

\begin{remark}
If $\hh=(n,\ldots,n)$, then the isomorphism \eqref{eq:cohomology_Hess(N,h)} is written as
\begin{align*}
H^*(\Fl(\C^n)) \cong \C[x_1,\ldots,x_n]/(f_{n,1},f_{n,2},\ldots,f_{n,n}).
\end{align*}
On the other hand, it is well-known by \cite{Bor53} that $H^*(\Fl(\C^n))$ is isomorphic to the quotient of the polynomial ring $\C[x_1,\ldots,x_n]$ by an ideal generated by the $i$-th elementary symmetric polynomials in the variables $x_1,\ldots,x_n$ for $i \in [n]$. 
We can verify that this ideal equals the ideal $(f_{n,1},f_{n,2},\ldots,f_{n,n})$ in $\C[x_1,\ldots,x_n]$ by \cite[Proof of Lemma~6.8]{AHHM}.
\end{remark}

By the recursive formula \eqref{eq:recursive_formula_fij} for $f_{i,j}$, we can add redundant generators in \eqref{eq:cohomology_Hess(N,h)} as follows:
\begin{align} \label{eq:cohomology_Hess(N,h)_2}
H^*(\Hess(\matN,\hh)) \cong \C[x_1,\ldots,x_n]/(f_{i,j} \mid i \geq \hh(j)). 
\end{align}
Tymoczko constructed an action of the symmetric group $S_n$ on the cohomology of regular semisimple Hessenberg varieties in \cite{Tym08}.
This action is called the \emph{dot action}. 
It is known that the dot action is related to the cohomology ring of regular nilpotent Hessenberg varieties.

\begin{theorem} $($\cite[Theorem~B]{AHHM}$)$ \label{theorem:AHHM_NandS}
Let $\hh: [n] \rightarrow [n]$ be a Hessenberg function.
Then there is an isomorphism of graded $\C$-algebras
\begin{align*} 
H^*(\Hess(\matN,\hh)) \cong H^*(\Hess(\matS,\hh))^{S_n},
\end{align*}
where $H^*(\Hess(\matS,\hh))^{S_n}$ denotes the invariants in $H^*(\Hess(\matS,\hh))$ under the dot action of the symmetric group $S_n$. 
\end{theorem}

For a Hessenberg function $\hh : [n] \to [n]$, we define the dual $\hh^*$ of $\hh$ by
\begin{align} \label{eq:HessenbergFunctionDual}
\hh^*(i) \coloneqq |\{j \in [n] \mid \hh(j) \geq n+1-i \}|
\end{align}
for $i \in [n]$. 
Let $\HessSp$ and $\HessSp^*$ be the Hessenberg spaces associated with $\hh$ and $\hh^*$, respectively. 
Then we have
\begin{align} \label{eq:Hdual}
\HessSp^* = \{w_0 {}^t\matA w_0 \mid \matA \in \HessSp \},
\end{align}
where ${}^t\matA$ denotes the transpose of a matrix $\matA$, and $w_0$ is the permutation matrix associated with the longest element of the symmetric group $S_n$.
In other words, $\HessSp^*$ is obtained from $\HessSp$ by flipping along the anti-diagonal line. 

\begin{example}
Considering $\hh=(3,4,4,5,5)$ given in Example~\ref{example:Hess_func_(3,4,4,5,5)}, its dual $\hh^*$ is given by $\hh^*=(2,4,5,5,5)$. 
The associated Hessenberg space $\HessSp^*$ is pictorially described in Figure~\ref{picture:dual_Hessenberg_function}.
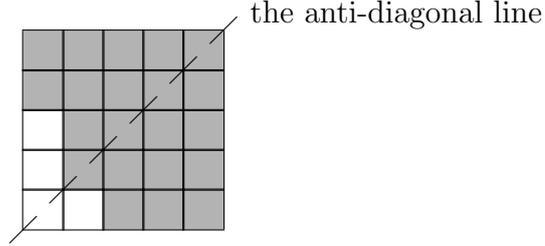
\begin{figure}[h]
\begin{center}
\begin{picture}(80,80)
\put(0,63){\colorbox{gray}}
\put(0,67){\colorbox{gray}}
\put(0,72){\colorbox{gray}}
\put(4,63){\colorbox{gray}}
\put(4,67){\colorbox{gray}}
\put(4,72){\colorbox{gray}}
\put(9,63){\colorbox{gray}}
\put(9,67){\colorbox{gray}}
\put(9,72){\colorbox{gray}}

\put(15,63){\colorbox{gray}}
\put(15,67){\colorbox{gray}}
\put(15,72){\colorbox{gray}}
\put(19,63){\colorbox{gray}}
\put(19,67){\colorbox{gray}}
\put(19,72){\colorbox{gray}}
\put(24,63){\colorbox{gray}}
\put(24,67){\colorbox{gray}}
\put(24,72){\colorbox{gray}}

\put(30,63){\colorbox{gray}}
\put(30,67){\colorbox{gray}}
\put(30,72){\colorbox{gray}}
\put(34,63){\colorbox{gray}}
\put(34,67){\colorbox{gray}}
\put(34,72){\colorbox{gray}}
\put(39,63){\colorbox{gray}}
\put(39,67){\colorbox{gray}}
\put(39,72){\colorbox{gray}}

\put(45,63){\colorbox{gray}}
\put(45,67){\colorbox{gray}}
\put(45,72){\colorbox{gray}}
\put(49,63){\colorbox{gray}}
\put(49,67){\colorbox{gray}}
\put(49,72){\colorbox{gray}}
\put(54,63){\colorbox{gray}}
\put(54,67){\colorbox{gray}}
\put(54,72){\colorbox{gray}}

\put(60,63){\colorbox{gray}}
\put(60,67){\colorbox{gray}}
\put(60,72){\colorbox{gray}}
\put(64,63){\colorbox{gray}}
\put(64,67){\colorbox{gray}}
\put(64,72){\colorbox{gray}}
\put(69,63){\colorbox{gray}}
\put(69,67){\colorbox{gray}}
\put(69,72){\colorbox{gray}}

\put(0,48){\colorbox{gray}}
\put(0,52){\colorbox{gray}}
\put(0,57){\colorbox{gray}}
\put(4,48){\colorbox{gray}}
\put(4,52){\colorbox{gray}}
\put(4,57){\colorbox{gray}}
\put(9,48){\colorbox{gray}}
\put(9,52){\colorbox{gray}}
\put(9,57){\colorbox{gray}}

\put(15,48){\colorbox{gray}}
\put(15,52){\colorbox{gray}}
\put(15,57){\colorbox{gray}}
\put(19,48){\colorbox{gray}}
\put(19,52){\colorbox{gray}}
\put(19,57){\colorbox{gray}}
\put(24,48){\colorbox{gray}}
\put(24,52){\colorbox{gray}}
\put(24,57){\colorbox{gray}}

\put(30,48){\colorbox{gray}}
\put(30,52){\colorbox{gray}}
\put(30,57){\colorbox{gray}}
\put(34,48){\colorbox{gray}}
\put(34,52){\colorbox{gray}}
\put(34,57){\colorbox{gray}}
\put(39,48){\colorbox{gray}}
\put(39,52){\colorbox{gray}}
\put(39,57){\colorbox{gray}}

\put(45,48){\colorbox{gray}}
\put(45,52){\colorbox{gray}}
\put(45,57){\colorbox{gray}}
\put(49,48){\colorbox{gray}}
\put(49,52){\colorbox{gray}}
\put(49,57){\colorbox{gray}}
\put(54,48){\colorbox{gray}}
\put(54,52){\colorbox{gray}}
\put(54,57){\colorbox{gray}}

\put(60,48){\colorbox{gray}}
\put(60,52){\colorbox{gray}}
\put(60,57){\colorbox{gray}}
\put(64,48){\colorbox{gray}}
\put(64,52){\colorbox{gray}}
\put(64,57){\colorbox{gray}}
\put(69,48){\colorbox{gray}}
\put(69,52){\colorbox{gray}}
\put(69,57){\colorbox{gray}}


\put(15,33){\colorbox{gray}}
\put(15,37){\colorbox{gray}}
\put(15,42){\colorbox{gray}}
\put(19,33){\colorbox{gray}}
\put(19,37){\colorbox{gray}}
\put(19,42){\colorbox{gray}}
\put(24,33){\colorbox{gray}}
\put(24,37){\colorbox{gray}}
\put(24,42){\colorbox{gray}}

\put(30,33){\colorbox{gray}}
\put(30,37){\colorbox{gray}}
\put(30,42){\colorbox{gray}}
\put(34,33){\colorbox{gray}}
\put(34,37){\colorbox{gray}}
\put(34,42){\colorbox{gray}}
\put(39,33){\colorbox{gray}}
\put(39,37){\colorbox{gray}}
\put(39,42){\colorbox{gray}}

\put(45,33){\colorbox{gray}}
\put(45,37){\colorbox{gray}}
\put(45,42){\colorbox{gray}}
\put(49,33){\colorbox{gray}}
\put(49,37){\colorbox{gray}}
\put(49,42){\colorbox{gray}}
\put(54,33){\colorbox{gray}}
\put(54,37){\colorbox{gray}}
\put(54,42){\colorbox{gray}}

\put(60,33){\colorbox{gray}}
\put(60,37){\colorbox{gray}}
\put(60,42){\colorbox{gray}}
\put(64,33){\colorbox{gray}}
\put(64,37){\colorbox{gray}}
\put(64,42){\colorbox{gray}}
\put(69,33){\colorbox{gray}}
\put(69,37){\colorbox{gray}}
\put(69,42){\colorbox{gray}}


\put(15,18){\colorbox{gray}}
\put(15,22){\colorbox{gray}}
\put(15,27){\colorbox{gray}}
\put(19,18){\colorbox{gray}}
\put(19,22){\colorbox{gray}}
\put(19,27){\colorbox{gray}}
\put(24,18){\colorbox{gray}}
\put(24,22){\colorbox{gray}}
\put(24,27){\colorbox{gray}}

\put(30,18){\colorbox{gray}}
\put(30,22){\colorbox{gray}}
\put(30,27){\colorbox{gray}}
\put(34,18){\colorbox{gray}}
\put(34,22){\colorbox{gray}}
\put(34,27){\colorbox{gray}}
\put(39,18){\colorbox{gray}}
\put(39,22){\colorbox{gray}}
\put(39,27){\colorbox{gray}}

\put(45,18){\colorbox{gray}}
\put(45,22){\colorbox{gray}}
\put(45,27){\colorbox{gray}}
\put(49,18){\colorbox{gray}}
\put(49,22){\colorbox{gray}}
\put(49,27){\colorbox{gray}}
\put(54,18){\colorbox{gray}}
\put(54,22){\colorbox{gray}}
\put(54,27){\colorbox{gray}}

\put(60,18){\colorbox{gray}}
\put(60,22){\colorbox{gray}}
\put(60,27){\colorbox{gray}}
\put(64,18){\colorbox{gray}}
\put(64,22){\colorbox{gray}}
\put(64,27){\colorbox{gray}}
\put(69,18){\colorbox{gray}}
\put(69,22){\colorbox{gray}}
\put(69,27){\colorbox{gray}}



\put(30,3){\colorbox{gray}}
\put(30,7){\colorbox{gray}}
\put(30,12){\colorbox{gray}}
\put(34,3){\colorbox{gray}}
\put(34,7){\colorbox{gray}}
\put(34,12){\colorbox{gray}}
\put(39,3){\colorbox{gray}}
\put(39,7){\colorbox{gray}}
\put(39,12){\colorbox{gray}}

\put(45,3){\colorbox{gray}}
\put(45,7){\colorbox{gray}}
\put(45,12){\colorbox{gray}}
\put(49,3){\colorbox{gray}}
\put(49,7){\colorbox{gray}}
\put(49,12){\colorbox{gray}}
\put(54,3){\colorbox{gray}}
\put(54,7){\colorbox{gray}}
\put(54,12){\colorbox{gray}}

\put(60,3){\colorbox{gray}}
\put(60,7){\colorbox{gray}}
\put(60,12){\colorbox{gray}}
\put(64,3){\colorbox{gray}}
\put(64,7){\colorbox{gray}}
\put(64,12){\colorbox{gray}}
\put(69,3){\colorbox{gray}}
\put(69,7){\colorbox{gray}}
\put(69,12){\colorbox{gray}}

\put(0,0){\framebox(15,15)}
\put(15,0){\framebox(15,15)}
\put(30,0){\framebox(15,15)}
\put(45,0){\framebox(15,15)}
\put(60,0){\framebox(15,15)}
\put(0,15){\framebox(15,15)}
\put(15,15){\framebox(15,15)}
\put(30,15){\framebox(15,15)}
\put(45,15){\framebox(15,15)}
\put(60,15){\framebox(15,15)}
\put(0,30){\framebox(15,15)}
\put(15,30){\framebox(15,15)}
\put(30,30){\framebox(15,15)}
\put(45,30){\framebox(15,15)}
\put(60,30){\framebox(15,15)}
\put(0,45){\framebox(15,15)}
\put(15,45){\framebox(15,15)}
\put(30,45){\framebox(15,15)}
\put(45,45){\framebox(15,15)}
\put(60,45){\framebox(15,15)}
\put(0,60){\framebox(15,15)}
\put(15,60){\framebox(15,15)}
\put(30,60){\framebox(15,15)}
\put(45,60){\framebox(15,15)}
\put(60,60){\framebox(15,15)}

\put(-5,-5){\line(1,1){10}}
\put(10,10){\line(1,1){10}}
\put(25,25){\line(1,1){10}}
\put(40,40){\line(1,1){10}}
\put(55,55){\line(1,1){10}}
\put(70,70){\line(1,1){10}}

\put(85,78){the anti-diagonal line}
\end{picture}
\end{center}
\vspace{-10pt}
\caption{The configuration corresponding to $\hh^*=(2,4,5,5,5)$.}
\label{picture:dual_Hessenberg_function}
\end{figure} 
\end{example}

\begin{lemma} \label{lemma:dual_homeo}
Let $\hh: [n] \rightarrow [n]$ be a Hessenberg function and $\hh^*$ its dual.
Then $\Hess(\matN,\hh)$ is homeomorphic to $\Hess(\matN,\hh^*)$. 
In particular, we have an isomorphism 
\begin{align*}
H^*(\Hess(\matN,\hh)) \cong H^*(\Hess(\matN,\hh^*))
\end{align*}
as graded $\C$-algebras. 
\end{lemma}

\begin{proof}
Recall that the flag variety $\GL_n(\C)/B$ is identified with $\U(n)/T$ where $\U(n)$ is the unitary group of degree $n$ and $T$ is the set of diagonal matrices in $\U(n)$. 
Note that $w_0 \, {}^t\matN \, w_0 =\matN$. 
By \eqref{eq:Hdual} we have 
\begin{align*}
u^{-1} \matN u \in \HessSp \iff {}^{t}(w_0 u w_0) \, \matN \, {}^{t}(w_0 u^{-1} w_0) \in \HessSp^*
\end{align*}
for any $u \in \U(n)$.
Thus the map $\U(n)/T \to \U(n)/T$ by sending $uT$ to $(w_0 \bar{u} w_0 ) T$ yields an homeomorphism from $\Hess(\matN,\hh)$ to $\Hess(\matN,\hh^*)$ where $\bar{u}$ denotes the matrix with complex conjugated entries of $u$.
\end{proof}

\section{Coordinate rings and quantization} \label{sect:quantization and coordinate rings}

Recall that the set $U$ in \eqref{eq:open set U} is naturally identified with an open set around the identity permutation $e \in S_n$ in the flag variety $\GL_n(\C)/B$. 
We denote the coordinate functions on $U$ by $z_{ij}$, i.e. 
\begin{align*}
z_{ij}(g) = x_{ij} \ \ \ \textrm{for} \ g \in U.
\end{align*} 
Consider the open subscheme $\ZZ(\matN,\hh)_e$ by the restriction of $\ZZ(\matN,\hh)$ to the open set $U$. 
To be more precise, by setting 
\begin{align} \label{eq:nuij}
\nu_{i,j}(g) \coloneqq (g^{-1} \matN g)_{ij}
\end{align} 
for $g \in U$, we can write
\begin{align*} 
\ZZ(\matN,\hh)_e = \Spec \C[z_{ij} \mid 1 \leq j < i \leq n]/(\nu_{i,j} \mid i > \hh(j)).  
\end{align*}
Equivalently, the set of global sections $\Gamma(\ZZ(\matN,\HessSp)_e, \mathcal{O}_{\ZZ(\matN,\HessSp)_e})$ is 
\begin{align} \label{eq:global_section_regular_nilpotent} 
\Gamma(\ZZ(\matN,\hh)_e, \mathcal{O}_{\ZZ(\matN,\hh)_e}) \cong \C[z_{ij} \mid 1 \leq j < i \leq n]/(\nu_{i,j} \mid i > \hh(j)).
\end{align}
Note that if $\hh(j) >j$ for all $j \in [n-1]$, then $\Gamma(\ZZ(\matN,\hh)_e, \mathcal{O}_{\ZZ(\matN,\hh)_e})$ is the coordinate ring of the (set-theoretic) intersection $\Hess(\matN,\hh) \cap U$ by Proposition~\ref{proposition:reduced}.
We define 
\begin{align*} 
\deg z_{ij} = 2(i-j)
\end{align*}
for $1 \leq j < i \leq n$.
Then $\nu_{i,j}$ defined in \eqref{eq:nuij} is a homogeneous polynomial of degree $2(i-j+1)$ in the polynomial ring $\C[z_{ij} \mid 1 \leq j < i \leq n]$ by \cite[Lemma~6.4]{HorShi}.
Hence, $\Gamma(\ZZ(\matN,\hh)_e, \mathcal{O}_{\ZZ(\matN,\hh)_e})
$ in \eqref{eq:global_section_regular_nilpotent} has a grading.
Peterson found a connection between $\Gamma(\ZZ(\matN,\hh)_e, \mathcal{O}_{\ZZ(\matN,\hh)_e})$ for $\hh=(2,3,4,\ldots,n,n)$ and the quantum cohomology of $\Fl(\C^n)$, as explained below.

Consider the polynomial ring $\C[x_1,\ldots,x_n,q_{rs} \mid 1 \leq r < s \leq n]$ equipped with a grading defined by \eqref{eq:degree xi} and 
\begin{align*} 
\deg q_{rs} = 2(s-r+1) 
\end{align*} 
for all $1 \leq r < s \leq n$.
We call the variables $q_{rs}$ \emph{quantum parameters}.
Let $M_n$ be the matrix defined by
\begin{align*}
M_n \coloneqq \left(
 \begin{array}{@{\,}ccccc@{\,}}
     x_1 & q_{12} & q_{13} & \cdots & q_{1n} \\
     -1 & x_2 & q_{23} & \cdots & q_{2n} \\ 
      0 & \ddots & \ddots & \ddots & \vdots \\ 
      \vdots & \ddots & -1 & x_{n-1} & q_{n-1 \, n} \\
      0 & \cdots & 0 & -1 & x_n 
 \end{array}
 \right).
\end{align*}
We then define the \emph{$q_{rs}$-quantized elementary symmetric polynomials} $E_1^{(n)}, \ldots, E_n^{(n)}$ in the polynomial ring $\C[x_1,\ldots,x_n,q_{rs} \mid 1 \leq r < s \leq n]$ by 
\begin{align*}
\det(\lambda I_n - M_n) = \lambda^n - E_1^{(n)} \lambda^{n-1} + E_2^{(n)} \lambda^{n-2} + \cdots + (-1)^n E_n^{(n)},
\end{align*}
where $I_n$ is the identity matrix of order $n$. 
Namely, $E_i^{(n)}$ is the coefficient of $\lambda^{n-i}$ for $\det(\lambda I_n - M_n)$ multiplied by $(-1)^i$. 
The recursive formula for $E_i^{(n)}$ is described as follows (\cite[Lemma~4.8]{HorShi}):
\begin{align} \label{eq:recursive quantized elementary symmetric polynomials}
E_i^{(n)} = E_i^{(n-1)} + E_{i-1}^{(n-1)} x_n + \sum_{k=1}^{i-1} E_{i-1-k}^{(n-1-k)} q_{n-k \, n}  \ \ \ \textrm{for} \ 1 \leq i \leq n \ 
\end{align}
with the convention that $E_0^{(*)}=1$ for arbitrary $*$, $\sum_{k=1}^{i-1} E_{i-1-k}^{(n-1-k)} q_{n-k \, n}=0$ whenever $i=1$, and $E_{n}^{(n-1)}=0$ whenever $i=n$. 
It is straightforward from the recursive formula above to see that $E_i^{(n)}$ is a homogeneous polynomial of degree $2i$ in the polynomial ring $\C[x_1,\ldots,x_n,q_{rs} \mid 1 \leq r < s \leq n]$ by \cite[Lemma~6.2]{HorShi}.
An explicit formula for $E_i^{(n)}$ is described as follows.

\begin{lemma} \label{lemma:Ein}
For a consecutive string $J=[r,s]\coloneqq \{r, r+1, \ldots, s \}$ with $r \leq s$, we set $\rho_J = q_{rs}$ with the convention that $\rho_J = x_r$ whenever $r=s$.
Then we have
\begin{align*} 
E_i^{(n)} = \sum \rho_{J_1} \rho_{J_2} \cdots \rho_{J_m} \ \ \ \textrm{for} \ 1 \leq i \leq n,
\end{align*}
where the sum runs over all consecutive substrings $J_1, J_2, \ldots, J_m$ of $[n]$ such that $\bigcap_{k=1}^m J_k = \emptyset$ and $\sum_{k=1}^m |J_k| = i$. 
\end{lemma}

\begin{proof}
We prove this by induction on $n$. 
The base case when $n=1$ is clear.
Assume that $n > 1$ and the claim holds for arbitrary $n'$ with $n' <n$.
Then the recursive formula in \eqref{eq:recursive quantized elementary symmetric polynomials} with the inductive assumption on $n$ yields the claim.
\end{proof}

\begin{example}
For $n \leq 3$, the $q_{rs}$-quantized elementary symmetric polynomials $E_i^{(n)}$ are described as follows:
\begin{align*}
&E_1^{(1)} = \rho_{\{1\}} =x_1, \ E_1^{(2)} = \rho_{\{1\}} + \rho_{\{2\}} = x_1+x_2, \\
&E_1^{(3)} = \rho_{\{1\}} + \rho_{\{2\}} + \rho_{\{3\}} = x_1+x_2+x_3, \ E_2^{(2)} = \rho_{\{1\}}\rho_{\{2\}} + \rho_{\{1,2\}} = x_1x_2+q_{12}, \\ 
&E_2^{(3)} = \rho_{\{1\}}\rho_{\{2\}} + \rho_{\{1\}}\rho_{\{3\}} + \rho_{\{2\}}\rho_{\{3\}} + \rho_{\{1,2\}} + \rho_{\{2,3\}} = x_1x_2+x_1x_3+x_2x_3+q_{12}+q_{23}, \\
&E_3^{(3)} = \rho_{\{1\}}\rho_{\{2\}}\rho_{\{3\}} + \rho_{\{1\}}\rho_{\{2,3\}} + \rho_{\{1,2\}}\rho_{\{3\}} + \rho_{\{1,2,3\}}= x_1x_2x_3+q_{23}x_1+q_{12}x_3+q_{13}. 
\end{align*}
\end{example} 

We set 
\begin{align} \label{eq:Qn}
Q_n \coloneqq \C[x_1,\ldots,x_n,q_{rs} \mid 1 \leq r < s \leq n]/(E_1^{(n)}, \ldots, E_n^{(n)}).
\end{align}
By slight abuse of notation, we denote by $F \in Q_n$ the image of $F$ in the polynomial ring $\C[x_1,\ldots,x_n,q_{rs} \mid 1 \leq r < s \leq n]$ under the natural surjection $\C[x_1,\ldots,x_n,q_{rs} \mid 1 \leq r < s \leq n] \rightarrow Q_n$. 
We define the map 
\begin{align} \label{eq:varphi}
\varphi: \C[z_{ij} \mid 1 \leq j < i \leq n] \to Q_n; \ \ \ z_{ij} \mapsto E_{i-j}^{(n-j)}.
\end{align}

\begin{theorem} $($ \cite[Theorem~4.13 and Proposition~5.2]{HorShi} $)$ \label{theorem:HorShi}
The map $\varphi$ in \eqref{eq:varphi} is an isomorphism of graded $\C$-algebras. 
We also have
\begin{align*}
\varphi^{-1}(q_{rs}) = -\nu_{n+1-r, \, n+1-s} 
\end{align*}
for $1 \leq r < s \leq n$.
In particular, for any Hessenberg function $\hh: [n] \to [n]$, we obtain an isomorphism 
\begin{align} \label{eq:Coordinate_ring_regular_nilpotent}
\Gamma(\ZZ(\matN,\hh)_e, \mathcal{O}_{\ZZ(\matN,\hh)_e}) \cong \frac{\C[x_1,\ldots,x_n, q_{rs} \mid 2 \leq s \leq n, n-\hh(n+1-s)<r<s]}{({}^{\hh}E_1^{(n)}, \ldots, {}^{\hh}E_n^{(n)})}, 
\end{align}
as graded $\C$-algebras where ${}^{\hh}E_i^{(n)}$ is defined by 
\begin{align*}
{}^{\hh}E_i^{(n)} \coloneqq E_i^{(n)}|_{q_{rs}=0 \ (2 \leq s \leq n \ \textrm{and} \ 1 \leq r \leq n-\hh(n+1-s))}. 
\end{align*}
\end{theorem}

\begin{remark}
One can pictorially see which variables $q_{rs}$ in the definition of ${}^{\hh}E_i^{(n)}$ are zero.
Consider the matrix $w_0 M_n w_0$ where $w_0$ is the longest element of the symmetric group $S_n$ (identified with the associated permutation matrix).
Then $w_0 M_n w_0 \in \HessSp$ if and only if $q_{rs}=0$ for all $2 \leq s \leq n$ and $1 \leq r \leq n-\hh(n+1-s)$ where $\HessSp$ is the Hessenberg space associated with $\hh$. 
For example, if $\hh=(3,4,4,5,5)$, then ${}^{\hh}E_i^{(5)}$ is a polynomial obtained from $E_i^{(5)}$ by setting $q_{13}=q_{14}=q_{15}=q_{25}=0$, as shown in Figure~\ref{picture:qrs}.
\begin{figure}[h]
\begin{center}
\begin{picture}(75,75)
\put(0,63){\colorbox{gray}}
\put(0,67){\colorbox{gray}}
\put(0,72){\colorbox{gray}}
\put(4,63){\colorbox{gray}}
\put(4,67){\colorbox{gray}}
\put(4,72){\colorbox{gray}}
\put(9,63){\colorbox{gray}}
\put(9,67){\colorbox{gray}}
\put(9,72){\colorbox{gray}}

\put(15,63){\colorbox{gray}}
\put(15,67){\colorbox{gray}}
\put(15,72){\colorbox{gray}}
\put(19,63){\colorbox{gray}}
\put(19,67){\colorbox{gray}}
\put(19,72){\colorbox{gray}}
\put(24,63){\colorbox{gray}}
\put(24,67){\colorbox{gray}}
\put(24,72){\colorbox{gray}}

\put(30,63){\colorbox{gray}}
\put(30,67){\colorbox{gray}}
\put(30,72){\colorbox{gray}}
\put(34,63){\colorbox{gray}}
\put(34,67){\colorbox{gray}}
\put(34,72){\colorbox{gray}}
\put(39,63){\colorbox{gray}}
\put(39,67){\colorbox{gray}}
\put(39,72){\colorbox{gray}}

\put(45,63){\colorbox{gray}}
\put(45,67){\colorbox{gray}}
\put(45,72){\colorbox{gray}}
\put(49,63){\colorbox{gray}}
\put(49,67){\colorbox{gray}}
\put(49,72){\colorbox{gray}}
\put(54,63){\colorbox{gray}}
\put(54,67){\colorbox{gray}}
\put(54,72){\colorbox{gray}}

\put(60,63){\colorbox{gray}}
\put(60,67){\colorbox{gray}}
\put(60,72){\colorbox{gray}}
\put(64,63){\colorbox{gray}}
\put(64,67){\colorbox{gray}}
\put(64,72){\colorbox{gray}}
\put(69,63){\colorbox{gray}}
\put(69,67){\colorbox{gray}}
\put(69,72){\colorbox{gray}}

\put(0,48){\colorbox{gray}}
\put(0,52){\colorbox{gray}}
\put(0,57){\colorbox{gray}}
\put(4,48){\colorbox{gray}}
\put(4,52){\colorbox{gray}}
\put(4,57){\colorbox{gray}}
\put(9,48){\colorbox{gray}}
\put(9,52){\colorbox{gray}}
\put(9,57){\colorbox{gray}}

\put(15,48){\colorbox{gray}}
\put(15,52){\colorbox{gray}}
\put(15,57){\colorbox{gray}}
\put(19,48){\colorbox{gray}}
\put(19,52){\colorbox{gray}}
\put(19,57){\colorbox{gray}}
\put(24,48){\colorbox{gray}}
\put(24,52){\colorbox{gray}}
\put(24,57){\colorbox{gray}}

\put(30,48){\colorbox{gray}}
\put(30,52){\colorbox{gray}}
\put(30,57){\colorbox{gray}}
\put(34,48){\colorbox{gray}}
\put(34,52){\colorbox{gray}}
\put(34,57){\colorbox{gray}}
\put(39,48){\colorbox{gray}}
\put(39,52){\colorbox{gray}}
\put(39,57){\colorbox{gray}}

\put(45,48){\colorbox{gray}}
\put(45,52){\colorbox{gray}}
\put(45,57){\colorbox{gray}}
\put(49,48){\colorbox{gray}}
\put(49,52){\colorbox{gray}}
\put(49,57){\colorbox{gray}}
\put(54,48){\colorbox{gray}}
\put(54,52){\colorbox{gray}}
\put(54,57){\colorbox{gray}}

\put(60,48){\colorbox{gray}}
\put(60,52){\colorbox{gray}}
\put(60,57){\colorbox{gray}}
\put(64,48){\colorbox{gray}}
\put(64,52){\colorbox{gray}}
\put(64,57){\colorbox{gray}}
\put(69,48){\colorbox{gray}}
\put(69,52){\colorbox{gray}}
\put(69,57){\colorbox{gray}}

\put(0,33){\colorbox{gray}}
\put(0,37){\colorbox{gray}}
\put(0,42){\colorbox{gray}}
\put(4,33){\colorbox{gray}}
\put(4,37){\colorbox{gray}}
\put(4,42){\colorbox{gray}}
\put(9,33){\colorbox{gray}}
\put(9,37){\colorbox{gray}}
\put(9,42){\colorbox{gray}}

\put(15,33){\colorbox{gray}}
\put(15,37){\colorbox{gray}}
\put(15,42){\colorbox{gray}}
\put(19,33){\colorbox{gray}}
\put(19,37){\colorbox{gray}}
\put(19,42){\colorbox{gray}}
\put(24,33){\colorbox{gray}}
\put(24,37){\colorbox{gray}}
\put(24,42){\colorbox{gray}}

\put(30,33){\colorbox{gray}}
\put(30,37){\colorbox{gray}}
\put(30,42){\colorbox{gray}}
\put(34,33){\colorbox{gray}}
\put(34,37){\colorbox{gray}}
\put(34,42){\colorbox{gray}}
\put(39,33){\colorbox{gray}}
\put(39,37){\colorbox{gray}}
\put(39,42){\colorbox{gray}}

\put(45,33){\colorbox{gray}}
\put(45,37){\colorbox{gray}}
\put(45,42){\colorbox{gray}}
\put(49,33){\colorbox{gray}}
\put(49,37){\colorbox{gray}}
\put(49,42){\colorbox{gray}}
\put(54,33){\colorbox{gray}}
\put(54,37){\colorbox{gray}}
\put(54,42){\colorbox{gray}}

\put(60,33){\colorbox{gray}}
\put(60,37){\colorbox{gray}}
\put(60,42){\colorbox{gray}}
\put(64,33){\colorbox{gray}}
\put(64,37){\colorbox{gray}}
\put(64,42){\colorbox{gray}}
\put(69,33){\colorbox{gray}}
\put(69,37){\colorbox{gray}}
\put(69,42){\colorbox{gray}}


\put(15,18){\colorbox{gray}}
\put(15,22){\colorbox{gray}}
\put(15,27){\colorbox{gray}}
\put(19,18){\colorbox{gray}}
\put(19,22){\colorbox{gray}}
\put(19,27){\colorbox{gray}}
\put(24,18){\colorbox{gray}}
\put(24,22){\colorbox{gray}}
\put(24,27){\colorbox{gray}}

\put(30,18){\colorbox{gray}}
\put(30,22){\colorbox{gray}}
\put(30,27){\colorbox{gray}}
\put(34,18){\colorbox{gray}}
\put(34,22){\colorbox{gray}}
\put(34,27){\colorbox{gray}}
\put(39,18){\colorbox{gray}}
\put(39,22){\colorbox{gray}}
\put(39,27){\colorbox{gray}}

\put(45,18){\colorbox{gray}}
\put(45,22){\colorbox{gray}}
\put(45,27){\colorbox{gray}}
\put(49,18){\colorbox{gray}}
\put(49,22){\colorbox{gray}}
\put(49,27){\colorbox{gray}}
\put(54,18){\colorbox{gray}}
\put(54,22){\colorbox{gray}}
\put(54,27){\colorbox{gray}}

\put(60,18){\colorbox{gray}}
\put(60,22){\colorbox{gray}}
\put(60,27){\colorbox{gray}}
\put(64,18){\colorbox{gray}}
\put(64,22){\colorbox{gray}}
\put(64,27){\colorbox{gray}}
\put(69,18){\colorbox{gray}}
\put(69,22){\colorbox{gray}}
\put(69,27){\colorbox{gray}}




\put(45,3){\colorbox{gray}}
\put(45,7){\colorbox{gray}}
\put(45,12){\colorbox{gray}}
\put(49,3){\colorbox{gray}}
\put(49,7){\colorbox{gray}}
\put(49,12){\colorbox{gray}}
\put(54,3){\colorbox{gray}}
\put(54,7){\colorbox{gray}}
\put(54,12){\colorbox{gray}}

\put(60,3){\colorbox{gray}}
\put(60,7){\colorbox{gray}}
\put(60,12){\colorbox{gray}}
\put(64,3){\colorbox{gray}}
\put(64,7){\colorbox{gray}}
\put(64,12){\colorbox{gray}}
\put(69,3){\colorbox{gray}}
\put(69,7){\colorbox{gray}}
\put(69,12){\colorbox{gray}}

\put(0,0){\framebox(15,15){\tiny $q_{15}$}}
\put(15,0){\framebox(15,15){\tiny $q_{14}$}}
\put(30,0){\framebox(15,15){\tiny $q_{13}$}}
\put(45,0){\framebox(15,15){\tiny $q_{12}$}}
\put(60,0){\framebox(15,15){\tiny $x_1$}}
\put(0,15){\framebox(15,15){\tiny $q_{25}$}}
\put(15,15){\framebox(15,15){\tiny $q_{24}$}}
\put(30,15){\framebox(15,15){\tiny $q_{23}$}}
\put(45,15){\framebox(15,15){\tiny $x_2$}}
\put(60,15){\framebox(15,15)}
\put(0,30){\framebox(15,15){\tiny $q_{35}$}}
\put(15,30){\framebox(15,15){\tiny $q_{34}$}}
\put(30,30){\framebox(15,15){\tiny $x_3$}}
\put(45,30){\framebox(15,15)}
\put(60,30){\framebox(15,15)}
\put(0,45){\framebox(15,15){\tiny $q_{45}$}}
\put(15,45){\framebox(15,15){\tiny $x_4$}}
\put(30,45){\framebox(15,15)}
\put(45,45){\framebox(15,15)}
\put(60,45){\framebox(15,15)}
\put(0,60){\framebox(15,15){\tiny $x_5$}}
\put(15,60){\framebox(15,15)}
\put(30,60){\framebox(15,15)}
\put(45,60){\framebox(15,15)}
\put(60,60){\framebox(15,15)}
\end{picture}
\end{center}
\vspace{-10pt}
\caption{The variables in $w_0 M_5 w_0$ and $\hh=(3,4,4,5,5)$.}
\label{picture:qrs}
\end{figure}
\end{remark}

\begin{remark}
We set $q_{rs}=0$ for $s-r >1$ and $q_{s \, s+1}=q_s$. 
Then $E_i^{(n)}$ is called the \emph{quantized elementary symmetric polynomial} in the polynomial ring $\C[x_1,\ldots,x_n,q_1,\ldots,q_{n-1}]$.
It is known that the quotient ring $Q_n$ in \eqref{eq:Qn} in this setting is isomorphic to the quantum cohomology ring of $\Fl(\C^n)$ by \cite{Font95, GK}.
It is the right hand side of \eqref{eq:Coordinate_ring_regular_nilpotent} for $\hh=(2,3,4,\ldots,n,n)$, so $\Gamma(\ZZ(\matN,\hh)_e, \mathcal{O}_{\ZZ(\matN,\hh)_e})$ when $\hh=(2,3,4,\ldots,n,n)$ is isomorphic to the quantum cohomology ring of the flag variety $\Fl(\C^n)$. 
This surprising connection was found by Peterson in an unpublished paper \cite{Pet}.
We refer the reader to \cite{Kos, Rie} for the result.
For this reason, the regular nilpotent Hessenberg variety $\Hess(\matN,\hh)$ for $\hh=(2,3,4,\ldots,n,n)$ is called the \emph{Peterson variety}.
\end{remark}

\section{The main theorem} \label{sect:the main theorem}

As in the case of regular nilpotent Hessenberg schemes, we consider the open subscheme $\ZZ(\matS,\hh)_e$ of the regular semisimple Hessenberg scheme $\ZZ(\matS,\hh)$ by restricting to the open set $U$. 
Namely, by setting 
\begin{align} \label{eq:xiij}
\xi_{i,j}(g) \coloneqq (g^{-1} \matS g)_{ij}
\end{align} 
for $g \in U$, one has
\begin{align*} 
\ZZ(\matS,\hh)_e = \Spec \C[z_{ij} \mid 1 \leq j < i \leq n]/(\xi_{i,j} \mid i > \hh(j)).  
\end{align*}
In other words, the set of global sections $\Gamma(\ZZ(\matS,\HessSp)_e, \mathcal{O}_{\ZZ(\matS,\HessSp)_e})$ is 
\begin{align} \label{eq:global_section_regular_semisimple} 
\Gamma(\ZZ(\matS,\hh)_e, \mathcal{O}_{\ZZ(\matN,\hh)_e}) \cong \C[z_{ij} \mid 1 \leq j < i \leq n]/(\xi_{i,j} \mid i > \hh(j)).
\end{align}
We now connect $\Gamma(\ZZ(\matS,\hh)_e, \mathcal{O}_{\ZZ(\matN,\hh)_e})$ with the quantization of $f_{i,j}$ defined in \eqref{eq:fij}.

For a positive integer $n$, we denote by $e_i^{(n)}$ the $i$-th elementary symmetric polynomial in the variables $x_1,\ldots,x_n$, i.e. 
\begin{align*}
e_i^{(n)} =e_i(x_1,\ldots,x_n) = \sum_{1 \leq p_1 < \dots < p_i \leq n} x_{p_1} \cdots x_{p_i}. 
\end{align*} 
Here, we take the convention that $e_0^{(n)}=1$ for all $n > 0$, and $e_i^{(n)}=0$ unless $0 \leq i \leq n$. 
Note that $E_i^{(n)}$ specializes to $e_i^{(n)}$ in the case $q_{rs}=0$ for all $1 \leq r < s \leq n$.
For $i_1,\ldots,i_m$ with $0 \leq i_k \leq k$, we set 
\begin{align*}
e_{i_1,\ldots,i_m} \coloneqq e_{i_1}^{(1)} \cdots e_{i_m}^{(m)}. 
\end{align*}
The polynomial $e_{i_1,\ldots,i_m}$ is called a \emph{standard elementary monomial}. 

\begin{proposition} $($\cite[Proposition~3.3]{FGP}$)$ \label{proposition:standard elementary monomials}
The standard elementary monomials form an additive basis of the polynomial ring $\Z[x_1,x_2,\ldots]$ in infinitely many variables.
\end{proposition}

We define a \emph{quantum standard elementary monomial} by  
\begin{align*}
E_{i_1,\ldots,i_m} \coloneqq E_{i_1}^{(1)} \cdots E_{i_m}^{(m)}. 
\end{align*}
Let $f_{i,j}$ be the polynomial in \eqref{eq:fij} for $i \geq j \geq 1$. 
By Proposition~\ref{proposition:standard elementary monomials} we can uniquely write $f_{i,j}$ as a linear combination of standard elementary monomials 
\begin{align*}
f_{i,j} = \sum_{i_1,\ldots,i_m} c_{i_1,\ldots,i_m} e_{i_1,\ldots,i_m} \ \ \ (c_{i_1,\ldots,i_m} \in \Z).
\end{align*}
Then we define the \emph{quantization $F_{i,j}$ of} $f_{i,j}$ is defined by 
\begin{align} \label{eq:Fij}
F_{i,j} = \sum_{i_1,\ldots,i_m} c_{i_1,\ldots,i_m} E_{i_1,\ldots,i_m}.
\end{align}

\begin{remark}
The definition of $F_{i,j}$ is an analogue of the definition for Fomin--Gelfand--Postnikov's quantum Schubert polynomials in \cite{FGP}, as explained in Introduction.
\end{remark}

Since $f_{i,j}$ is homogeneous of degree $2(i-j+1)$, the quantization $F_{i,j}$ is also homogeneous of degree $2(i-j+1)$. 
Note that the quantization $F_{i,j}$ specializes to $f_{i,j}$ when the quantum parameters are set to $0$.
Our main theorem is as follows.

\begin{theorem} \label{theorem:main}
Let $\varphi$ be the isomorphism of graded $\C$-algebras in \eqref{eq:varphi}. 
Let $F_{i,j}$ be the quantization of $f_{i,j}$ defined in \eqref{eq:Fij} for $i \geq j \geq 1$.
Then the followings hold.
\begin{enumerate}
\item The polynomial $F_{i,j}$ belongs to $\C[x_1,\ldots,x_i, q_{rs} \mid 1 \leq r < s \leq i]$ for all $i \geq j \geq 1$. 
\item The ideal generated by $E_1^{(n)}, \ldots, E_n^{(n)}$ is equal to the ideal generated by $F_{n,1}, \ldots, F_{n,n}$ in the polynomial ring $\C[x_1,\ldots,x_n,q_{rs} \mid 1 \leq r < s \leq n]$. 
In particular, the quotient ring $Q_n$ in \eqref{eq:Qn} is written as 
\begin{align*}
Q_n = \C[x_1,\ldots,x_n,q_{rs} \mid 1 \leq r < s \leq n]/(F_{n,1}, \ldots, F_{n,n}).
\end{align*}
\item We obtain 
\begin{align*}
\varphi^{-1}(F_{i,j}) = (-1)^{i-j} \xi_{n+1-j, \, n-i}
\end{align*}
for any $1 \leq j \leq i \leq n-1$.
In particular, for arbitrary Hessenberg function $\hh: [n] \to [n]$, the isomorphism $\varphi$ induces an isomorphism of graded $\C$-algebras
\begin{align} \label{eq:Coordinate_ring_regular_semisimple}
\Gamma(\ZZ(\matS,\hh)_e, \mathcal{O}_{\ZZ(\matS,\hh)_e}) \cong \C[x_1,\ldots,x_n,q_{rs} \mid 1 \leq r < s \leq n]/(F_{i,j} \mid i \geq \hh^*(j)), 
\end{align}
where $\hh^*$ is the dual of $\hh$ defined in \eqref{eq:HessenbergFunctionDual}. 
\end{enumerate}
\end{theorem}

We conclude the following consequences. 

\begin{corollary} \label{corollary:main1}
Let $\hh: [n] \to [n]$ be a Hessenberg function. 
Then there is an isomorphism of graded $\C$-algebras
\begin{align*} 
\Gamma(\ZZ(\matS,\hh)_e, \mathcal{O}_{\ZZ(\matS,\hh)_e})/(\nu_{i,j} \mid i>j) \cong H^*(\Hess(\matN,\hh)),
\end{align*}
where $\nu_{i,j}$ is defined in \eqref{eq:nuij}. 
\end{corollary}

\begin{proof}
It follows from Theorems~\ref{theorem:HorShi} and \ref{theorem:main} that 
\begin{align*}
\Gamma(\ZZ(\matS,\hh)_e, \mathcal{O}_{\ZZ(\matS,\hh)_e})/(\nu_{i,j} \mid i>j) \cong \C[x_1,\ldots,x_n]/(f_{i,j} \mid i \geq \hh^*(j)),
\end{align*} 
which is isomorphic to $H^*(\Hess(\matN,\hh^*))$ by \eqref{eq:cohomology_Hess(N,h)_2}.
The result follows from Lemma~\ref{lemma:dual_homeo}. 
\end{proof}

\begin{corollary}
Let $\hh: [n] \to [n]$ be a Hessenberg function. 
Then there is an isomorphism of graded $\C$-algebras
\begin{align*} 
\Gamma(\ZZ(\matS,\hh)_e, \mathcal{O}_{\ZZ(\matS,\hh)_e})/(\nu_{i,j} \mid i>j) \cong H^*(\Hess(\matS,\hh))^{S_n}. 
\end{align*}
\end{corollary}

\begin{proof}
It follows from Thereom~\ref{theorem:AHHM_NandS} and Corollary~\ref{corollary:main1}. 
\end{proof}

\section{Determinant formula for $f_{i,j}$} \label{sect:deterinant formula for fij}

The aim of this section is to describe the expansion of $f_{i,j}$ in the basis of standard elementary monomials by using determinant.

Recall that the symmetric group $S_n$ acts on the polynomial ring $\Z[x_1,\ldots,x_n]$ by
\begin{align*}
\big(w(f) \big)(x_1,\ldots,x_n) = f(x_{w(1)},\ldots,x_{w(n)})
\end{align*}
for $w \in S_n$ and $f=f(x_1,\ldots,x_n) \in \Z[x_1,\ldots,x_n]$.
Let $s_i \in S_n$ be the transposition of $i$ and $i+1$ for each $i \in [n-1]$. 
Then the divided difference operator $\partial_i$ on $\Z[x_1,\ldots,x_n]$ is defined by
\begin{align*}
\partial_i(f) = \frac{f-s_i(f)}{x_i-x_{i+1}}
\end{align*}
for $f \in \Z[x_1,\ldots,x_n]$.
Since $f-s_i(f)$ is divisible by $x_i-x_{i+1}$, $\partial_i(f)$ is a polynomial. 
In particular, $\partial_i(f)=0$ if and only if $f$ is symmetric in $x_i$ and $x_{i+1}$.
If $f$ is homogeneous of degree $2r$, then $\partial_i(f)$ is homogeneous of degree $2(r-1)$.
The divided difference operator $\partial_i$ satisfies the following ``Leibniz formula''
\begin{align*}
\partial_i(f \cdot g) = \partial_i(f) \cdot g + s_i(f) \cdot \partial_i(g)
\end{align*}
for $f,g \in \Z[x_1,\ldots,x_n]$.
In particular, $\partial_i$ commutes with multiplication by any polynomial which is symmetric in $x_i$ and $x_{i+1}$.

\begin{lemma} \label{lemma:fij_divided_difference_operator} 
Let $f_{i,j}$ be the polynomial in \eqref{eq:fij}. 
For $i > j \geq 1$, we have
\begin{align*}
\partial_k (f_{i,j})=\begin{cases}
-f_{i-1,j} \ & {\rm if} \ k=i \\
f_{i,j+1} \ & {\rm if} \ k=j \\
0 \ & {\rm otherwise} \ \\
\end{cases}
\end{align*}
\end{lemma}

\begin{proof}
The claims for the cases when $k=i$ and $k=j$ were proved in \cite[Proposition~2.1]{Hor18}.
If $k \neq i, j$, then it is straightforward to see that $s_k(f_{i,j})=f_{i,j}$ by \eqref{eq:fij}.
Hence, we obtain $\partial_k (f_{i,j})=0$ in this case.
\end{proof}

For $n>0$, the $i$-th complete symmetric polynomial $h_i^{(n)}$ in the variables $x_1,\ldots,x_n$ is given by
\begin{align*}
h_i^{(n)} = h_i(x_1,\ldots,x_n)= \sum_{1 \leq p_1 \leq \dots \leq p_i \leq n} x_{p_1} \cdots x_{p_i} 
\end{align*} 
with the convention that $h_0^{(n)}=1$ and $h_i^{(n)}=0$ if $i<0$. 
It is known a fundamental relation between the elementary symmetric polynomials and the complete symmetric polynomials as follows (e.g. \cite[Section~6.1]{Ful97}):
\begin{align} \label{eq:e_kh_k_relation}
\sum_{i=0}^m (-1)^i e_i^{(n)} h_{m-i}^{(n)} = 0 \ \ \ \textrm{for} \ m>0.
\end{align}
Our first goal is to describe $f_{i,j}$ in terms of elementary symmetric polynomials and complete symmetric polynomials.

\begin{lemma} \label{lemma:fi1}
For $i \geq 1$, one has
\begin{align*}
f_{i,1} = \sum_{k=0}^{i-1} (-1)^k (i-k) e_k^{(i)} x_1^{i-k}.
\end{align*}
\end{lemma}

\begin{proof}
If $i=1$, then the claim is clear. In what follows, we suppose that $i>1$. 
By \eqref{eq:fij} we have
\begin{align} \label{eq:lemma_proof_fi1}
f_{i,1} = (x_1-x_2)(x_1-x_3)\cdots(x_1-x_i)x_1 = \sum_{\ell=0}^{i-1} (-1)^\ell e_\ell(x_2,\ldots,x_i) x_1^{i-\ell}. 
\end{align}
We here show that 
\begin{align} \label{eq:lemma_proof_e_ell_2i}
e_\ell(x_2,\ldots,x_i) = \sum_{p=0}^\ell (-1)^p e_{\ell-p}^{(i)} x_1^p 
\end{align}
by indeuction on $\ell$.
The base case when $\ell=0$ is clear. 
Assume that $\ell > 0$ and \eqref{eq:lemma_proof_e_ell_2i} is true for $\ell-1$.
Then we have
\begin{align*}
e_\ell(x_2,\ldots,x_i) &= e_\ell(x_1,\ldots,x_i) - e_{\ell-1}(x_2,\ldots,x_i) x_1 \\
&= e_\ell^{(i)} - \sum_{p=0}^{\ell-1} (-1)^p e_{\ell-1-p}^{(i)} x_1^{p+1} \ \ \ (\textrm{by the inductive assumption}) \\
&=e_\ell^{(i)} - \sum_{p=1}^\ell (-1)^{p-1} e_{\ell-p}^{(i)} x_1^p =\sum_{p=0}^\ell (-1)^p e_{\ell-p}^{(i)} x_1^p. 
\end{align*}
We proved \eqref{eq:lemma_proof_e_ell_2i}.
By \eqref{eq:lemma_proof_fi1} and \eqref{eq:lemma_proof_e_ell_2i}, one has
\begin{align*} 
f_{i,1} = \sum_{\ell=0}^{i-1} \sum_{p=0}^\ell (-1)^{\ell+p} e_{\ell-p}^{(i)} x_1^{i-\ell+p} = \sum_{k=0}^{i-1} (-1)^k (i-k) e_k^{(i)} x_1^{i-k},
\end{align*}
as desired.
\end{proof}

\begin{lemma} \label{lemma:divided_difference_complete_symmetric}
For $i \geq 1$ and $k \geq 1$, we have
\begin{align*}
\partial_i(h_k^{(i)}) = h_{k-1}^{(i+1)}.
\end{align*}
\end{lemma}

\begin{proof}
In general, it is known that 
\begin{align} \label{eq:h_k_relation}
h_k(x_1,\ldots,x_n) = \sum_{\ell=0}^{k} h_{k-\ell}(x_1,\ldots,x_m) \cdot h_\ell(x_{m+1},\ldots,x_n)
\end{align}
for any $m \in [n-1]$.
In fact, it follows from the following equality
\begin{align*}
\sum_{k \geq 0} h_k(x_1,\ldots,x_n) t^k &= \prod_{p=1}^{n} \frac{1}{1-x_pt} =\left( \prod_{p=1}^{m} \frac{1}{1-x_pt} \right) \cdot \left( \prod_{p=m+1}^{n} \frac{1}{1-x_pt} \right) \\
&= \left( \sum_{k_1 \geq 0} h_{k_1}(x_1,\ldots,x_m) t^{k_1} \right) \cdot \left( \sum_{k_2 \geq 0} h_{k_2}(x_{m+1},\ldots,x_n) t^{k_2} \right).
\end{align*}
In particular, since $h_k^{(i)} = \sum_{\ell=0}^{k} h_{k-\ell}(x_1,\ldots,x_{i-1}) \cdot x_i^\ell$ and $h_{k-\ell}(x_1,\ldots,x_{i-1})$ is symmetric in $x_i$ and $x_{i+1}$, we have
\begin{align*}
\partial_i(h_k^{(i)}) &= \sum_{\ell=0}^{k} h_{k-\ell}(x_1,\ldots,x_{i-1}) \cdot \partial_i(x_i^\ell) = \sum_{\ell=1}^{k} h_{k-\ell}(x_1,\ldots,x_{i-1}) \cdot h_{\ell-1}(x_i,x_{i+1}) \\
&= \sum_{\ell=0}^{k-1} h_{k-1-\ell}(x_1,\ldots,x_{i-1}) \cdot h_{\ell}(x_i,x_{i+1}) = h_{k-1}^{(i+1)},
\end{align*}
where we used \eqref{eq:h_k_relation} again for the last equality. 
\end{proof}

\begin{proposition} \label{proposition:fijekhk}
For any $i \geq j \geq 1$, it holds that
\begin{align*}
f_{i,j} = \sum_{k=0}^{i-j+1} (-1)^k (i-k) e_k^{(i)} h_{i-j+1-k}^{(j)}.
\end{align*}
\end{proposition}

\begin{proof}
We prove the claim by induction on $j$. 
The base case is $j = 1$. What we want to show is 
\begin{align*}
f_{i,1} = \sum_{k=0}^{i} (-1)^k (i-k) e_k^{(i)} x_1^{i-k} = \sum_{k=0}^{i-1} (-1)^k (i-k) e_k^{(i)} x_1^{i-k}, 
\end{align*} 
which is proved by Lemma~\ref{lemma:fi1}. 
We now assume $j>1$ and the claim is known for $j-1$.
It then follows from Lemmas~\ref{lemma:fij_divided_difference_operator} and \ref{lemma:divided_difference_complete_symmetric} that
\begin{align*}
f_{i,j} &= \partial_{j-1} (f_{i,j-1}) = \partial_{j-1} \left(\sum_{k=0}^{i-j+2} (-1)^k (i-k) e_k^{(i)} h_{i-j+2-k}^{(j-1)} \right) \ \ \ (\textrm{by the inductive hypothesis}) \\
&=\sum_{k=0}^{i-j+2} (-1)^k (i-k) e_k^{(i)} \partial_{j-1}(h_{i-j+2-k}^{(j-1)}) =\sum_{k=0}^{i-j+1} (-1)^k (i-k) e_k^{(i)} h_{i-j+1-k}^{(j)}. 
\end{align*}
\end{proof}

\begin{proposition} \label{proposition:fijekhkmatrix}
Let $1 \leq m \leq n$.
The following equality holds:
\begin{align} \label{eq:fijekhkmatrix}
&\begin{pmatrix}
f_{m,m} & 1 & 0 & \cdots & 0 \\
f_{m+1,m} & f_{m+1,m+1} & 2 & \ddots & \vdots \\
f_{m+2,m} & f_{m+2,m+1} & \ddots & \ddots & 0 \\
\vdots & \vdots &  & \ddots & n-m \\
f_{n,m} & f_{n,m+1} & \cdots & \cdots & f_{n,n}
\end{pmatrix} \\
= 
&\left( (-1)^{i+j}e_{i-j}^{(m-1+i)} \right)_{i,j \in [n-m+1]}
\begin{pmatrix}
1 &  & &   \\
     & 2 & & \\
     &   & \ddots &  \\
     &  & & n-m+1 \\ 
\end{pmatrix}
\left( h_{i-j+1}^{(m-1+j)} \right)_{i,j \in [n-m+1]}. \notag
\end{align} 
\end{proposition}

\begin{proof}
The $(i,j)$-th entry of the right hand side of \eqref{eq:fijekhkmatrix} is given by
\begin{align} \label{eq:proof_proposition_fijekhkmatrix}
\sum_{k=1}^{n-m+1} (-1)^{i+k} k e_{i-k}^{(m-1+i)} h_{k-j+1}^{(m-1+j)}. 
\end{align}
Both $e_{i-k}^{(m-1+i)}$ and $h_{k-j+1}^{(m-1+j)}$ are non-zero if and only if $j-1 \leq k \leq i$, so \eqref{eq:proof_proposition_fijekhkmatrix} equals $0$ whenever $i < j-1$. 
For $i \geq j-1$, \eqref{eq:proof_proposition_fijekhkmatrix} is written as 
\begin{align*} 
\sum_{k=j-1}^{i} (-1)^{i+k} k e_{i-k}^{(m-1+i)} h_{k-j+1}^{(m-1+j)} &= \sum_{k=0}^{i-j+1} (-1)^{i+j-1+k} (k+j-1) e_{i-j+1-k}^{(m-1+i)} h_{k}^{(m-1+j)} \\
&= \sum_{k=0}^{i-j+1} (-1)^{k} (i-k) e_{k}^{(m-1+i)} h_{i-j+1-k}^{(m-1+j)}. 
\end{align*}
Hence, what we want to show is 
\begin{align} \label{eq:proof_proposition_fijekhkmatrix2}
\sum_{k=0}^{i-j+1} (-1)^{k} (i-k) e_{k}^{(m-1+i)} h_{i-j+1-k}^{(m-1+j)} = \begin{cases}
f_{m-1+i,m-1+j} \ &\textrm{if} \ i \geq j; \\
j-1 \ &\textrm{if} \ i = j-1. 
\end{cases} 
\end{align}
If $i = j-1$, then it is straightforward to see \eqref{eq:proof_proposition_fijekhkmatrix2}. 
If $i \geq j$, then it follows from Proposition~\ref{proposition:fijekhk} that
\begin{align*} 
f_{m-1+i,m-1+j} &= \sum_{k=0}^{i-j+1} (-1)^k (m-1+i-k) e_k^{(m-1+i)} h_{i-j+1-k}^{(m-1+j)} \\
&= \sum_{k=0}^{i-j+1} (-1)^k (i-k) e_k^{(m-1+i)} h_{i-j+1-k}^{(m-1+j)} + (m-1) \left( \sum_{k=0}^{i-j+1} (-1)^k e_k^{(m-1+i)} h_{i-j+1-k}^{(m-1+j)} \right).
\end{align*}
In order to prove \eqref{eq:proof_proposition_fijekhkmatrix2}, it suffices to show that 
\begin{align} \label{eq:proof_proposition_fijekhkmatrix3}
\sum_{k=0}^{i'-j'+1} (-1)^k e_k^{(i')} h_{i'-j'+1-k}^{(j')} = 0
\end{align}
for any $i' \geq j' > 0$. 
By comparing the coefficient of $t^{i'-j'+1}$ for the both sides of the following equality
\begin{align*}
\prod_{p=j'+1}^{i'} (1-tx_p) = \left( \prod_{p=1}^{i'} (1-tx_p) \right) \cdot \left( \prod_{p=1}^{j'} \frac{1}{1-tx_p} \right) = \sum_{\ell \geq 0} \left( \sum_{k = 0}^{\ell} (-1)^k e_k^{(i')} h_{\ell-k}^{(j')} \right) t^\ell,
\end{align*}
we obtain \eqref{eq:proof_proposition_fijekhkmatrix3}, and hence we proved \eqref{eq:proof_proposition_fijekhkmatrix2} as desired. 
\end{proof}

\begin{proposition} \label{proposition:determinanthij}
Let $1 \leq m \leq n$.
Then we have 
\begin{align*} 
\det \left( h_{i-j+1}^{(m-1+j)} \right)_{i,j \in [n-m+1]} = e_{n-m+1}^{(n)}. 
\end{align*} 
\end{proposition}

\begin{proof}
We prove this by descending induction on $m$. In other words, we induct on the order $n-m+1$ of the square matrix $\big( h_{i-j+1}^{(m-1+j)} \big)_{i,j \in [n-m+1]}$. 
The base case when $m=n$ is clear since $h_1^{(n)} = e_1^{(n)}$. 
Assume that $m<n$ and the claim is known for arbitrary $m'$ with $m' > m$.
By using the cofactor expansion along the $(n-m+1)$-th row for $\det \left( h_{i-j+1}^{(m-1+j)} \right)_{i,j \in [n-m+1]}$, we have
\begin{align*}
\det \left( h_{i-j+1}^{(m-1+j)} \right)_{i,j \in [n-m+1]} = \sum_{k=1}^{n-m+1} (-1)^{k-1} h_k^{(n-k+1)} e_{n-m+1-k}^{(n-k)}
\end{align*}
where we used the descending induction hypothesis on $m$.
Hence, it is enough to prove that 
\begin{align} \label{eq:proposition_proof_determinanthij_1}
\sum_{k=1}^{n-m+1} (-1)^{k-1} h_k^{(n-k+1)} e_{n-m+1-k}^{(n-k)} = e_{n-m+1}^{(n)}.
\end{align}
In order to prove \eqref{eq:proposition_proof_determinanthij_1}, we show the following identity
\begin{align} \label{eq:proposition_proof_determinanthij_2}
\sum_{k=1}^{\ell} (-1)^{k-1} h_k^{(n-k+1)} e_{n-m+1-k}^{(n-k)} = \sum_{k=1}^{\ell}  e_{n-m+1-k}^{(n-\ell)} \big( e_k(x_{n-\ell+1},\ldots,x_n)+(-1)^{k-1}h_k^{(n-\ell)} \big)
\end{align}
for $1 \leq \ell \leq n-m$.
We prove \eqref{eq:proposition_proof_determinanthij_2} by induction on $\ell$. 
The base case is $\ell=1$. In this case it is clear that $h_1^{(n)}e_{n-m}^{(n-1)}=e_{n-m}^{(n-1)}\big( e_1(x_n) + h_1^{(n-1)} \big)$. 
Now assume that $\ell >1$ and \eqref{eq:proposition_proof_determinanthij_2} holds for $\ell-1$. 
Then the left hand side of \eqref{eq:proposition_proof_determinanthij_2} is 
\begin{align} \label{eq:proposition_proof_determinanthij_3}
&\sum_{k=1}^{\ell-1} (-1)^{k-1} h_k^{(n-k+1)} e_{n-m+1-k}^{(n-k)} + (-1)^{\ell-1}h_\ell^{(n-\ell+1)}e_{n-m+1-\ell}^{(n-\ell)} \\
=&\sum_{k=1}^{\ell-1}  e_{n-m+1-k}^{(n-\ell+1)} \big( e_k(x_{n-\ell+2},\ldots,x_n)+(-1)^{k-1}h_k^{(n-\ell+1)} \big) + (-1)^{\ell-1}h_\ell^{(n-\ell+1)}e_{n-m+1-\ell}^{(n-\ell)} \notag
\end{align}
by the inductive assumption on $\ell$. 
By using the relation $e_{n-m+1-k}^{(n-\ell+1)} = e_{n-m-k}^{(n-\ell)} \cdot x_{n-\ell+1} + e_{n-m+1-k}^{(n-\ell)}$,  the right hand side of \eqref{eq:proposition_proof_determinanthij_3} is written as 
\begin{align} \label{eq:proposition_proof_determinanthij_4}
&\sum_{k=1}^{\ell-2} e_{n-m-k}^{(n-\ell)} \cdot x_{n-\ell+1} \big( e_k(x_{n-\ell+2},\ldots,x_n)+(-1)^{k-1}h_k^{(n-\ell+1)} \big) \\
& \hspace{5pt} +e_{n-m+1-\ell}^{(n-\ell)} \big( x_{n-\ell+1}e_{\ell-1}(x_{n-\ell+2},\ldots,x_n)+(-1)^{\ell-1} (-x_{n-\ell+1}h_{\ell-1}^{(n-\ell+1)} + h_{\ell}^{(n-\ell+1)}) \big) \notag \\ 
& \hspace{5pt} +\sum_{k=1}^{\ell-1} e_{n-m+1-k}^{(n-\ell)} \big(e_k(x_{n-\ell+2},\ldots,x_n)+(-1)^{k-1} h_k^{(n-\ell+1)}) \big) \notag \\
=&\sum_{k=2}^{\ell-1} e_{n-m+1-k}^{(n-\ell)} \cdot x_{n-\ell+1} \big( e_{k-1}(x_{n-\ell+2},\ldots,x_n)+(-1)^k h_{k-1}^{(n-\ell+1)} \big) \notag \\
& \hspace{5pt} +e_{n-m+1-\ell}^{(n-\ell)} \big(e_{\ell}(x_{n-\ell+1},\ldots,x_n)+(-1)^{\ell-1} h_\ell^{(n-\ell)} \big) \notag \\ 
& \hspace{5pt} +\sum_{k=1}^{\ell-1} e_{n-m+1-k}^{(n-\ell)} \big(e_k(x_{n-\ell+2},\ldots,x_n)+(-1)^{k-1} h_k^{(n-\ell+1)}) \big) \notag
\end{align}
where we used $x_{n-\ell+1}e_{\ell-1}(x_{n-\ell+2},\ldots,x_n) = x_{n-\ell+1}x_{n-\ell+2} \cdots x_n = e_\ell(x_{n-\ell+1},\ldots,x_n)$ and $h_{\ell}^{(n-\ell+1)} = h_{\ell-1}^{(n-\ell+1)} \cdot x_{n-\ell+1} + h_\ell^{(n-\ell)}$ for the second summand above.  
Then the right hand side of \eqref{eq:proposition_proof_determinanthij_4} equals  
\begin{align*} 
&\sum_{k=2}^{\ell-1} e_{n-m+1-k}^{(n-\ell)} \big( x_{n-\ell+1} \cdot e_{k-1}(x_{n-\ell+2},\ldots,x_n) + e_k(x_{n-\ell+2},\ldots,x_n) \\
& \hspace{15pt} +(-1)^{k-1}(-x_{n-\ell+1} \cdot h_{k-1}^{(n-\ell+1)} + h_k^{(n-\ell+1)}) \big) \notag \\
& \hspace{15pt} +e_{n-m+1-\ell}^{(n-\ell)} \big(e_{\ell}(x_{n-\ell+1},\ldots,x_n)+(-1)^{\ell-1} h_\ell^{(n-\ell)} \big) +e_{n-m}^{(n-\ell)} \big( e_1(x_{n-\ell+2},\ldots,x_n) + h_1^{(n-\ell+1)} \big) \notag \\
=&\sum_{k=2}^{\ell-1} e_{n-m+1-k}^{(n-\ell)} \big(e_{k}(x_{n-\ell+1},\ldots,x_n) +(-1)^{k-1}h_k^{(n-\ell)}  \big) \notag \\
& \hspace{15pt} +e_{n-m+1-\ell}^{(n-\ell)} \big(e_{\ell}(x_{n-\ell+1},\ldots,x_n)+(-1)^{\ell-1} h_\ell^{(n-\ell)} \big) +e_{n-m}^{(n-\ell)} \big( e_1(x_{n-\ell+1},\ldots,x_n) + h_1^{(n-\ell)} \big) \notag 
\end{align*}
which is the right hand side of \eqref{eq:proposition_proof_determinanthij_2} as desired.
We proved \eqref{eq:proposition_proof_determinanthij_2}.

Applying \eqref{eq:proposition_proof_determinanthij_2} for $\ell=n-m$, the left hand side of \eqref{eq:proposition_proof_determinanthij_1} is equal to
\begin{align*} 
&\sum_{k=1}^{n-m}  e_{n-m+1-k}^{(m)} \big( e_k(x_{m+1},\ldots,x_n)+(-1)^{k-1}h_k^{(m)} \big) + (-1)^{n-m} h_{n-m+1}^{(m)} \\
=&\sum_{k=1}^{n-m}  e_{n-m+1-k}^{(m)}e_k(x_{m+1},\ldots,x_n)+\sum_{k=1}^{n-m+1} (-1)^{k-1}e_{n-m+1-k}^{(m)}h_k^{(m)} \\
=&\sum_{k=1}^{n-m+1}  e_{n-m+1-k}^{(m)}e_k(x_{m+1},\ldots,x_n)+\sum_{k=1}^{n-m+1} (-1)^{k-1}e_{n-m+1-k}^{(m)}h_k^{(m)} 
\end{align*}
where we used $e_{n-m+1}(x_{m+1},\ldots,x_n)=0$ for the last equality. 
Note that this is written as 
\begin{align*} 
\sum_{k=0}^{n-m+1}  e_{n-m+1-k}^{(m)}e_k(x_{m+1},\ldots,x_n)+\sum_{k=0}^{n-m+1} (-1)^{k-1}e_{n-m+1-k}^{(m)}h_k^{(m)} 
\end{align*}
because the first summand for $k=0$ and the last summand for $k=0$ are cancelled.
The first summand equals $e_{n-m+1}^{(n)}$ by a similar argument of \eqref{eq:h_k_relation} and the second summand is $0$ by \eqref{eq:e_kh_k_relation}.
Hence, we proved \eqref{eq:proposition_proof_determinanthij_1} as desired.
\end{proof}

\begin{theorem} [Determinant formula for $e_k^{(i)}$]
Let $1 \leq k \leq i$. 
Then we have
\begin{align} \label{eq:determinant_formula_for_eki} 
e_k^{(i)}=\frac{1}{k!} 
\left|
 \begin{array}{@{\,}ccccc@{\,}}
f_{i-k+1,i-k+1} & 1 & 0 & \cdots & 0 \\
f_{i-k+2,i-k+1} & f_{i-k+2,i-k+2} & 2 & \ddots & \vdots \\
\vdots & \vdots & \ddots & \ddots & 0 \\
f_{i-1,i-k+1} & f_{i-1,i-k+2} & \cdots & f_{i-1,i-1} & k-1 \\
f_{i,i-k+1} & f_{i,i-k+2} & \cdots & f_{i,i-1} & f_{i,i}
\end{array}
\right|.
\end{align} 
\end{theorem}

\begin{proof}
Set $m=i-k+1$ and $n=i$ in  Proposition~\ref{proposition:fijekhkmatrix}. 
By taking the determinant of both sides of \eqref{eq:fijekhkmatrix}, the result follows from Proposition~\ref{proposition:determinanthij}. 
\end{proof}

\begin{theorem} \label{theorem:relations between eik and fij}
For $1 \leq k \leq i$, we have 
\begin{align} \label{eq:relations between eik and fij}
k e_k^{(i)} = \sum_{\ell=1}^k (-1)^{\ell-1} f_{i,i+1-\ell} e_{k-\ell}^{(i-\ell)}. 
\end{align}
\end{theorem}

\begin{proof}
By the cofactor expansion along the $k$-th row for the determinant appeared in the right hand side of \eqref{eq:determinant_formula_for_eki}, we have
\begin{align*}
k! e_k^{(i)} &= \sum_{\ell=1}^{k} (-1)^{\ell-1}f_{i,i+1-\ell} \cdot (k-\ell)!e_{k-\ell}^{(i-\ell)} \cdot (k-\ell+1)(k-\ell+2) \cdots (k-1) \\
&= (k-1)! \sum_{\ell=1}^{k} (-1)^{\ell-1}f_{i,i+1-\ell} e_{k-\ell}^{(i-\ell)}, 
\end{align*}
which proves the claim. 
\end{proof}

\begin{remark}
Let $p_k^{(n)} \coloneqq x_1^k+\cdots+x_n^k$ be the $k$-th power sum.
A fundamental relation between the elementary symmetric polynomials and the power sums is described as follows (e.g. \cite[Section~6.1]{Ful97}):
\begin{align} \label{eq:relations between eik and pki}
ke_k^{(i)} = \sum_{\ell=1}^{k} (-1)^{\ell-1} p_\ell^{(i)} e_{k-\ell}^{(i)}.
\end{align}
The relation \eqref{eq:relations between eik and fij} between \emph{truncated} elementary symmetric polynomials and polynomials $f_{i,j}$ is an analogue of \eqref{eq:relations between eik and pki}. 
\end{remark}

\begin{theorem} [Determinant formula for $f_{i,j}$] \label{theorem:determinant formula for fij}
For any $i \geq j \geq 1$, we have
\begin{align} \label{eq:determinant_formula_for_fij} 
f_{i,j}= 
\left|
 \begin{array}{@{\,}cccccc@{\,}}
e_1^{(j)} & 1 & 0 & 0 & \cdots & 0 \\
e_2^{(j+1)} & e_1^{(j+1)} & 1 & 0 & \ddots & \vdots \\
e_3^{(j+2)} & e_2^{(j+2)} & e_1^{(j+2)} & 1 & \ddots & \vdots \\
\vdots & \vdots & \vdots & \ddots & \ddots & 0 \\
e_{i-j}^{(i-1)} & e_{i-j-1}^{(i-1)} & e_{i-j-2}^{(i-1)} & \cdots & e_1^{(i-1)} & 1 \\
(i-j+1)e_{i-j+1}^{(i)} & (i-j)e_{i-j}^{(i)} & (i-j-1)e_{i-j-1}^{(i)} & \cdots & 2e_2^{(i)} & e_1^{(i)}
\end{array}
\right|.
\end{align} 
\end{theorem}

\begin{proof}
We prove \eqref{eq:determinant_formula_for_fij} by descending induction on $j$. 
As the base case $j=i$, one has $f_{j,j}=x_1+\cdots+x_j=e_1^{(j)}$ as desired.
Assume that $j<i$ and \eqref{eq:determinant_formula_for_fij} is true for arbitrary $j'$ with $j' > j$.  
By the cofactor expansion along the first column, the right hand side of \eqref{eq:determinant_formula_for_fij} equals
\begin{align} \label{eq:proof_theorem_determinant_formula_for_fij_1} 
(-1)^{i-j}(i-j+1) e_{i-j+1}^{(i)} + \sum_{\ell=1}^{i-j} (-1)^{\ell+i-j} f_{i,i+1-\ell} e_{i-j+1-\ell}^{(i-\ell)}, 
\end{align}
where we used the descending induction hypothesis on $j$.
On the other hand, we set $k=i-j+1$ in Theorem~\ref{theorem:relations between eik and fij}. Then we obtain 
\begin{align*}
(i-j+1) e_{i-j+1}^{(i)} = \sum_{\ell=1}^{i-j+1} (-1)^{\ell-1} f_{i,i+1-\ell} e_{i-j+1-\ell}^{(i-\ell)} = (-1)^{i-j}f_{i,j} + \sum_{\ell=1}^{i-j} (-1)^{\ell-1} f_{i,i+1-\ell} e_{i-j+1-\ell}^{(i-\ell)}. 
\end{align*}
In other words, we have 
\begin{align} \label{eq:proof_theorem_determinant_formula_for_fij_2}
f_{i,j}= (-1)^{i-j}(i-j+1) e_{i-j+1}^{(i)} + \sum_{\ell=1}^{i-j} (-1)^{\ell+i-j} f_{i,i+1-\ell} e_{i-j+1-\ell}^{(i-\ell)}. 
\end{align}
By \eqref{eq:proof_theorem_determinant_formula_for_fij_1} and \eqref{eq:proof_theorem_determinant_formula_for_fij_2}, we proved \eqref{eq:determinant_formula_for_fij}.
\end{proof}

\begin{corollary} \label{corollary:determinant_formula_for_Fij}
For any $i \geq j \geq 1$, we obtain
\begin{align} \label{eq:determinant_formula_for_Fij} 
F_{i,j}= 
\left|
 \begin{array}{@{\,}cccccc@{\,}}
E_1^{(j)} & 1 & 0 & 0 & \cdots & 0 \\
E_2^{(j+1)} & E_1^{(j+1)} & 1 & 0 & \ddots & \vdots \\
E_3^{(j+2)} & E_2^{(j+2)} & E_1^{(j+2)} & 1 & \ddots & \vdots \\
\vdots & \vdots & \vdots & \ddots & \ddots & 0 \\
E_{i-j}^{(i-1)} & E_{i-j-1}^{(i-1)} & E_{i-j-2}^{(i-1)} & \cdots & E_1^{(i-1)} & 1 \\
(i-j+1)E_{i-j+1}^{(i)} & (i-j)E_{i-j}^{(i)} & (i-j-1)E_{i-j-1}^{(i)} & \cdots & 2E_2^{(i)} & E_1^{(i)}
\end{array}
\right|.
\end{align} 
\end{corollary}

\begin{proof}
Theorem~\ref{theorem:determinant formula for fij} yields the expansion of $f_{i,j}$ in the basis of standard elementary monomials.
By the definition \eqref{eq:Fij}, we obtain the desired result \eqref{eq:determinant_formula_for_Fij}.
\end{proof}

\begin{remark}
By a similar discussion above, for $1 \leq k \leq i$, we have
\begin{align} 
k E_k^{(i)} &= \sum_{\ell=1}^k (-1)^{\ell-1} F_{i,i+1-\ell} E_{k-\ell}^{(i-\ell)}; \label{eq:relations between Eik and Fij} \\ 
E_k^{(i)}&=\frac{1}{k!} 
\left|
 \begin{array}{@{\,}ccccc@{\,}}
F_{i-k+1,i-k+1} & 1 & 0 & \cdots & 0 \\
F_{i-k+2,i-k+1} & F_{i-k+2,i-k+2} & 2 & \ddots & \vdots \\
\vdots & \vdots & \ddots & \ddots & 0 \\
F_{i-1,i-k+1} & F_{i-1,i-k+2} & \cdots & F_{i-1,i-1} & k-1 \\
F_{i,i-k+1} & F_{i,i-k+2} & \cdots & F_{i,i-1} & F_{i,i}
\end{array}
\right|. \label{eq:determinant_formula_for_Eki}
\end{align}
In fact, the cofactor expansion along the first column for the right hand side of \eqref{eq:determinant_formula_for_Fij} yields \eqref{eq:relations between Eik and Fij} where $j=i-k+1$. 
Also, the cofactor expansion along the $k$-th row for the right hand side of \eqref{eq:determinant_formula_for_Eki}, we obtain \eqref{eq:determinant_formula_for_Eki} by induction on $k$ and \eqref{eq:relations between Eik and Fij}. 
\end{remark}

\section{Proof of Theorem~\ref{theorem:main}} \label{sect:proof of main theorem}

In this section we prove our main theorem. 
The following lemma is useful to study the defining functions for $\ZZ(\matX,\hh)$. 

\begin{lemma} $($\cite[Lemma~2.3]{HorShi}$)$ \label{lemma:defining_functions} 
Let $\matX$ be an arbitrary element in $\gl_n(\C)$ and $i,j \in [n]$.
Then we have
\begin{align*}
(g^{-1} \matX g)_{ij} = \frac{\det (g_1 \, \ldots \, g_{i-1} \ \matX g_j \ g_{i+1} \, \ldots \, g_n)}{\det(g)}
\end{align*}
for $g \in \GL_n(\C)$ where $g_k$ is the $k$-th column vector of $g$ for each $k \in [n]$.
Note that the determinant in the numerator is obtained from $g$ by replacing the $i$-th column vector of $g$ to the $j$-th column vector of $\matX g$.
\end{lemma}

In particular, it follows from Lemma~\ref{lemma:defining_functions} that $\xi_{i,j}$ in \eqref{eq:xiij} is described as  
\begin{align*}
\xi_{i,j}(g) = \det (g_1 \, \ldots \, g_{i-1} \ \matS g_j \ g_{i+1} \, \ldots \, g_n)
\end{align*}
for $g \in U$. 
We here recall that $g \in U$ in \eqref{eq:open set U} is of the form 
\begin{align} \label{eq:g_lower_unipotent}
g=\left(
 \begin{array}{@{\,}ccccc@{\,}}
     1 &  &  &  &  \\
     x_{21} & 1 &  &  &  \\ 
     x_{31} & x_{32} & 1 &  &  \\ 
     \vdots& \vdots & \ddots & \ddots & \\
     x_{n1} & x_{n2} & \cdots & x_{n \, n-1} & 1 
 \end{array}
 \right), 
\end{align}
and $\matS$ is given in \eqref{eq:regular nilpotent semisimple}. 

\begin{lemma} \label{lemma:Sn-i-1}
Let $1 \leq j \leq i \leq n-1$. 
For arbitrary complex numbers $c_1,\ldots,c_{n-i-1}$, we set the $n \times n$ diaganol matrix $\matS_{n-i-1}$ with (ordered) diagonal entries $c_1,\ldots,c_{n-i-1},0,1,\ldots,i$.
Then $(g^{-1}\matS_{n-i-1}g)_{n+1-j \ n-i}$ is equal to 
\begin{align*}
(-1)^{i-j} \left|
 \begin{array}{@{\,}ccccc@{\,}}
  x_{n-i+1\ n-i} &  1 & 0 & \cdots & 0 \\
  2x_{n-i+2\ n-i} & x_{n-i+2\ n-i+1} & 1 & \ddots & \vdots \\ 
  3x_{n-i+3\ n-i} &  x_{n-i+3\ n-i+1} & x_{n-i+3\ n-i+2} & \ddots & 0  \\ 
  \vdots &  \vdots& \vdots & \ddots & 1  \\
  (i-j+1)x_{n-j+1\ n-i} &  x_{n-j+1\ n-i+1} & x_{n-j+1\ n-i+2} & \cdots & x_{n-j+1\ n-j}  
 \end{array}
 \right|
\end{align*}
where $g$ is of the form in \eqref{eq:g_lower_unipotent}.
\end{lemma}

\begin{proof}
By Lemma~\ref{lemma:defining_functions} we have 
\begin{align*}
(g^{-1}\matS_{n-i-1}g)_{n+1-j \ n-i} = \det (g_1 \, \ldots \, g_{n-j} \ \matS_{n-i-1} g_{n-i} \ g_{n-j+2} \, \ldots \, g_n).
\end{align*}
It is straightforward to see that $\det (g_1 \, \ldots \, g_{n-j} \ \matS_{n-i-1} g_{n-i} \ g_{n-j+2} \, \ldots \, g_n)$ is
\begin{align*}
\left|
 \begin{array}{@{\,}ccccc@{\,}}
   1 & 0 & \cdots & 0 & x_{n-i+1\ n-i} \\
  x_{n-i+2\ n-i+1} & 1 & \ddots & \vdots & 2x_{n-i+2\ n-i} \\ 
   x_{n-i+3\ n-i+1} & x_{n-i+3\ n-i+2} & \ddots & 0 & 3x_{n-i+3\ n-i} \\ 
   \vdots& \vdots & \ddots & 1 & \vdots \\
   x_{n-j+1\ n-i+1} & x_{n-j+1\ n-i+2} & \cdots & x_{n-j+1\ n-j} & (i-j+1)x_{n-j+1\ n-i}  
 \end{array}
 \right|. 
\end{align*}
We interchange adjacent columns such that the last column is interchanged with the column immediately to the left until it is exactly the first column. 
Then we have the desired result. 
\end{proof}

We now give a proof of Theorem~\ref{theorem:main}. 

\begin{proof}[Proof of Theorem~\ref{theorem:main}]
(1) It follows from Corollary~\ref{corollary:determinant_formula_for_Fij} that $F_{i,j}$ belongs to the polynomial ring $\C[x_1,\ldots,x_i, q_{rs} \mid 1 \leq r < s \leq i]$. 

(2) Set $i=n$ in Corollary~\ref{corollary:determinant_formula_for_Fij} and the cofactor expansion along the $(n-j+1)$-th row in \eqref{eq:determinant_formula_for_Fij} implies the following relations
\begin{align*}
F_{n,j} \equiv (-1)^{n-j} (n-j+1) E_{n-j+1}^{(n)} \pmod {E_1^{(n)}, \dots, E_{n-j}^{(n)}}
\end{align*}
for each $j \in [n]$.
This means that the ideal generated by $E_1^{(n)}, \ldots, E_{n-j+1}^{(n)}$ equals the ideal generated by $F_{n,j}, \ldots, F_{n,n}$ in the polynomial ring $\C[x_1,\ldots,x_n,q_{rs} \mid 1 \leq r < s \leq n]$ for all $j \in [n]$. 
The case when $j=1$ is the desired result. 

(3) We set $c_k = k-(n-i)$ for $k \in [n-i-1]$ in Lemma~\ref{lemma:Sn-i-1}. 
Then we have $\matS_{n-i-1} + (n-i) I_n =\matS$ where $I_n$ denotes the identity matrix of order $n$. 
Hence, one can see that 
\begin{align*}
(g^{-1} \matS_{n-i-1} g)_{k\ell} = (g^{-1} \matS g)_{k\ell}
\end{align*}
whenever $k \neq \ell$.
Combining this with Lemma~\ref{lemma:Sn-i-1}, we obtain 
\begin{align} \label{proof_main_3_1}
\xi_{n+1-j, \, n-i} = (-1)^{i-j} \left|
 \begin{array}{@{\,}ccccc@{\,}}
  z_{n-i+1\ n-i} &  1 & 0 & \cdots & 0 \\
  2z_{n-i+2\ n-i} & z_{n-i+2\ n-i+1} & 1 & \ddots & \vdots \\ 
  3z_{n-i+3\ n-i} &  z_{n-i+3\ n-i+1} & z_{n-i+3\ n-i+2} & \ddots & 0  \\ 
  \vdots &  \vdots& \vdots & \ddots & 1  \\
  (i-j+1)z_{n-j+1\ n-i} &  z_{n-j+1\ n-i+1} & z_{n-j+1\ n-i+2} & \cdots & z_{n-j+1\ n-j}  
 \end{array}
 \right|.
\end{align}
On the other hand, it follows from the definition of the isomorphism $\varphi$ in \eqref{eq:varphi} and Corollary~\ref{corollary:determinant_formula_for_Fij} that 
\begin{align} \label{proof_main_3_2}
\varphi^{-1}(F_{i,j}) = \left|
 \begin{array}{@{\,}ccccc@{\,}}
  z_{n-j+1\ n-j} &  1 & 0 & \cdots & 0 \\
  z_{n-j+1\ n-j-1} & z_{n-j\ n-j-1} & 1 & \ddots & \vdots \\ 
  \vdots & \vdots & \ddots & \ddots & 0  \\ 
  z_{n-j+1\ n-i+1} & z_{n-j\ n-i+1} & \cdots & z_{n-i+2\ n-i+1} & 1  \\
  (i-j+1)z_{n-j+1\ n-i} &  (i-j)z_{n-j\ n-i} & \cdots & 2z_{n-i+2\ n-i} & z_{n-i+1\ n-i}  
 \end{array}
 \right|.
\end{align}
Let $w_0^{(i-j+1)}$ be the permutation matrix associated with the longest element in $S_{i-j+1}$.
By taking the conjugation of the transpose matrix of the right hand side in \eqref{proof_main_3_2} by $w_0^{(i-j+1)}$,
we obtain 
\begin{align} \label{proof_main_3_3}
\varphi^{-1}(F_{i,j}) = (-1)^{i-j} \xi_{n+1-j, \, n-i} \ \ \ \textrm{for} \ 1 \leq j \leq i \leq n-1 
\end{align}
by \eqref{proof_main_3_1}. 

For the latter part, we take a Hessenberg function $\hh: [n] \to [n]$.  
It then follows from \eqref{proof_main_3_3} that the isomorphism $\varphi$ induces 
an isomorphism of graded $\C$-algebras
\begin{align} \label{proof_main_3_4}
&\Gamma(\ZZ(\matS,\hh)_e, \mathcal{O}_{\ZZ(\matS,\hh)_e}) \\
\cong &\C[x_1,\ldots,x_n,q_{rs} \mid 1 \leq r < s \leq n]/(F_{n,1},\ldots,F_{n,n})+(F_{n-j,n+1-i} \mid j \in [n-1], i > \hh(j)). \notag
\end{align}
By the definition of $\hh^*$ in \eqref{eq:HessenbergFunctionDual}, one can see that 
\begin{align*}
&\{(i,j) \in [n] \times [n] \mid j \in [n-1], i>\hh(j) \} \\
= &\{(i,j) \in [n] \times [n] \mid j \in [n-1], j \leq n-\hh^*(n+1-i) \}. 
\end{align*}
This implies that the ideal appeared in the right hand side of \eqref{proof_main_3_4} is written as 
\begin{align*}
&(F_{n,1},\ldots,F_{n,n})+(F_{n-j,n+1-i} \mid j \in [n-1], j \leq n-\hh^*(n+1-i)) \\
= &(F_{n,1},\ldots,F_{n,n})+(F_{i,j} \mid i \in [n-1], i \geq \hh^*(j)) \\ 
= &(F_{i,j} \mid i \geq \hh^*(j)). 
\end{align*}
Therefore, we conclude the isomorphism \eqref{eq:Coordinate_ring_regular_semisimple} as desired.
\end{proof}

\section{Quantized recursive formula for $F_{i,j}$} \label{sect:quantized recursive formula for Fij}

We finally give a qunatized version of the recursive formula in \eqref{eq:recursive_formula_fij}. 

\begin{theorem}[Quantized recursive formula for $F_{i,j}$] \label{theorem:quantized recursive formula for Fij}
Let $i \geq j \geq 1$ and $F_{i,j}$ the quantization of $f_{i,j}$ defined in \eqref{eq:Fij}.
Then the followings hold.
\begin{enumerate}
\item[(1)] For all $j \geq 1$, we have
\begin{align*} 
F_{j,j} = x_1+ \cdots + x_j. 
\end{align*}
\item[(2)] For any $i > j \geq 1$, we have
\begin{align*} 
\begin{split}
F_{i,j} &= F_{i-1,j-1} + \left( x_j \, F_{i-1,j} + \sum_{k=1}^{i-j-1} (-1)^k q_{j \, j+k} \, F_{i-1,j+k} \right) \\
& \hspace{25pt} - \left( x_i \, F_{i-1,j} + \sum_{k=1}^{i-j-1} (-1)^k q_{i-k \, i} \, F_{i-k-1,j} \right) +(-1)^{i-j}(i-j+1) q_{ji}
\end{split}
\end{align*}
with the convention $F_{*,0}=0$ for arbitrary $*$, and $\sum_{k=1}^{i-j-1} (-1)^k q_{j \, j+k} \, F_{i-1,j+k} = \sum_{k=1}^{i-j-1} (-1)^k q_{i-k \, i} \, F_{i-k-1,j} =0$ whenever $i=j+1$.
\end{enumerate}
\end{theorem}

\begin{remark}
The recursive formula given in Theorem~\ref{theorem:quantized recursive formula for Fij} specializes to the recursive formula in \eqref{eq:recursive_formula_fij} when we set the quantum parameters to $0$.
\end{remark}

\begin{example}
For $1 \leq j \leq i \leq 3$, the quantizations $F_{i,j}$ of $f_{i,j}$ are computed as follows:
\begin{align*}
F_{1,1} &= x_1, \ F_{2,2} = x_1+x_2, \ F_{3,3} = x_1+x_2+x_3, \\
F_{2,1} &= x_1 F_{1,1}-x_2 F_{1,1}-2q_{12} =(x_1-x_2)x_1-2q_{12}, \\ 
F_{3,2} &= F_{2,1}+x_2 F_{2,2}-x_3 F_{2,2}-2q_{23} =(x_1-x_3)x_1+(x_2-x_3)x_2-2q_{12}-2q_{23}, \\
F_{3,1} &= (x_1 F_{2,1} - q_{12} F_{2,2})-(x_3 F_{2,1}-q_{23}F_{1,1})+3q_{13} \\
&=(x_1-x_2)(x_1-x_3)x_1 +q_{12}(-3x_1-x_2+2x_3)+q_{23}x_1+3q_{13}. 
\end{align*}
\end{example}

\begin{proof}[Proof of Theorem~\ref{theorem:quantized recursive formula for Fij}]
(1) It is clear since $F_{j,j} = E_1^{(j)} =x_1+\cdots+x_j$ from Lemma~\ref{lemma:Ein}.

(2) Recall that $F_{i,j}$ is given by \eqref{eq:determinant_formula_for_Fij}. 
For the determinant appeared in \eqref{eq:determinant_formula_for_Fij}, we subtract the sum of $k$-th raw multiplied by $q_{j+k \ i}$ for $1 \leq k \leq i-j$ from the $(i-j+1)$-th raw where we regard $q_{ii} = x_i$. 
Then the result for $(i-j+2-\ell)$-th column in the $(i-j+1)$-th raw when $1 \leq \ell \leq i-j$ is given by
\begin{align*}
\ell E_\ell^{(i)} - \left(E_{\ell-1}^{(i-1)}x_i +\sum_{k=1}^{\ell-1} E_{\ell-k-1}^{(i-k-1)} q_{i-k \, i} \right) = \ell E_\ell^{(i-1)} + (\ell-1)\left(E_{\ell-1}^{(i-1)}x_i +\sum_{k=1}^{\ell-1} E_{\ell-k-1}^{(i-k-1)} q_{i-k \, i} \right) 
\end{align*}
where we used the recursive formula \eqref{eq:recursive quantized elementary symmetric polynomials}. 
Similarly, the result for first column in the $(i-j+1)$-th raw is described as
\begin{align*}
(i-j+1)(E_{i-j+1}^{(i-1)} + q_{ji}) + (i-j)\left(E_{i-j}^{(i-1)}x_i +\sum_{k=1}^{i-j-1} E_{i-j-k}^{(i-k-1)} q_{i-k \, i} \right). 
\end{align*}
Therefore, we obtain that 
\begin{align} \label{eq_proof_quantized_recursive_formula_for_Fij_1}
F_{i,j} = A + x_i \cdot B + \sum_{k=1}^{i-j-1} q_{i-k \, i} \cdot C_k
\end{align}
where $A,B$, and $C_k$ are the following three types of determinants:
\begin{align*}
A &\coloneqq \left|
 \begin{array}{@{\,}cccccc@{\,}}
E_1^{(j)} & 1 & 0 & 0 & \cdots & 0 \\
E_2^{(j+1)} & E_1^{(j+1)} & 1 & 0 & \ddots & \vdots \\
E_3^{(j+2)} & E_2^{(j+2)} & E_1^{(j+2)} & 1 & \ddots & \vdots \\
\vdots & \vdots & \vdots & \ddots & \ddots & 0 \\
E_{i-j}^{(i-1)} & E_{i-j-1}^{(i-1)} & E_{i-j-2}^{(i-1)} & \cdots & E_1^{(i-1)} & 1 \\
(i-j+1)(E_{i-j+1}^{(i-1)}+q_{ji}) & (i-j)E_{i-j}^{(i-1)} & (i-j-1)E_{i-j-1}^{(i-1)} & \cdots & 2E_2^{(i-1)} & E_1^{(i-1)}
\end{array}
\right|; \\
B &\coloneqq \left|
 \begin{array}{@{\,}cccccc@{\,}}
E_1^{(j)} & 1 & 0 & 0 & \cdots & 0 \\
E_2^{(j+1)} & E_1^{(j+1)} & 1 & 0 & \ddots & \vdots \\
E_3^{(j+2)} & E_2^{(j+2)} & E_1^{(j+2)} & 1 & \ddots & \vdots \\
\vdots & \vdots & \vdots & \ddots & \ddots & 0 \\
E_{i-j}^{(i-1)} & E_{i-j-1}^{(i-1)} & E_{i-j-2}^{(i-1)} & \cdots & E_1^{(i-1)} & 1 \\
(i-j)E_{i-j}^{(i-1)} & (i-j-1)E_{i-j-1}^{(i-1)} & (i-j-2)E_{i-j-2}^{(i-1)} & \cdots & E_1^{(i-1)} & 0 
\end{array}
\right|; \\
C_k &\coloneqq \left|
 \begin{array}{@{\,}cccccccc@{\,}}
E_1^{(j)} & 1 & 0 & \cdots & \cdots & \cdots & \cdots & 0 \\
E_2^{(j+1)} & E_1^{(j+1)} & 1 & 0 & & & & \vdots \\
\vdots & \vdots &\ddots & \ddots & \ddots & & & \vdots \\
E_{i-j-k}^{(i-k-1)} & E_{i-j-k-1}^{(i-k-1)} & \cdots & E_1^{(i-k-1)} & 1 & 0 & \cdots & 0 \\
E_{i-j-k+1}^{(i-k)} & E_{i-j-k}^{(i-k)} & \cdots & E_2^{(i-k)} & E_1^{(i-k)} & 1 & \ddots & \vdots \\
\vdots &\vdots & &\vdots & \vdots & \ddots & \ddots & 0 \\
E_{i-j}^{(i-1)} & E_{i-j-1}^{(i-1)} & \cdots & E_{k+1}^{(i-1)} & E_k^{(i-1)} & \cdots & E_1^{(i-1)} & 1 \\
(i-j)E_{i-j-k}^{(i-k-1)} & (i-j-1)E_{i-j-k-1}^{(i-k-1)} & \cdots & (k+1)E_1^{(i-k-1)} & k & 0 & \cdots & 0 
\end{array}
\right|. 
\end{align*}
In what follows, we prove the following equalities
\begin{align}
&B= - \, F_{i-1,j}; \label{eq_proof_quantized_recursive_formula_for_Fij_B} \\ 
&C_k= (-1)^{k+1} F_{i-k-1,j}; \label{eq_proof_quantized_recursive_formula_for_Fij_Ck} \\
&A =  F_{i-1,j-1} + \left( x_j \, F_{i-1,j} + \sum_{k=1}^{i-j-1} (-1)^k q_{j \, j+k} \, F_{i-1,j+k} +(-1)^{i-j}(i-j+1)q_{ji} \right). \label{eq_proof_quantized_recursive_formula_for_Fij_A} 
\end{align}
\noindent
\underline{\textbf{Proof of \eqref{eq_proof_quantized_recursive_formula_for_Fij_B}:}}
By the cofactor expansion along the last column for $B$, we have 
\begin{align*}
B = -F_{i-1,j} 
\end{align*}
by Corollary~\ref{corollary:determinant_formula_for_Fij}.
We proved \eqref{eq_proof_quantized_recursive_formula_for_Fij_B}. 

\noindent
\underline{\textbf{Proof of \eqref{eq_proof_quantized_recursive_formula_for_Fij_Ck}:}}
We interchange adjacent rows in $C_k$ such that the last row is interchanged with the row immediately above until it is exactly the $(i-j-k+1)$-th row. 
Then we have 
\begin{align*}
C_k = (-1)^k \left|
\begin{array}{@{\,}ccccc@{\,}}
E_1^{(j)} & 1 & 0 & \cdots & 0 \\
E_2^{(j+1)} & E_1^{(j+1)} & 1 & \ddots & \vdots \\
\vdots & \vdots &\ddots & \ddots & 0  \\
E_{i-j-k}^{(i-k-1)} & E_{i-j-k-1}^{(i-k-1)} & \cdots & E_1^{(i-k-1)} & 1 \\
(i-j)E_{i-j-k}^{(i-k-1)} & (i-j-1)E_{i-j-k-1}^{(i-k-1)} & \cdots & (k+1)E_1^{(i-k-1)} & k  
\end{array}
\right|. 
\end{align*}
By subtracting $(i-j-k)$-th raw multiplied by $k$ from the $(i-j-k+1)$-th raw, the right hand side above equals 
\begin{align*}
(-1)^k \left|
\begin{array}{@{\,}ccccc@{\,}}
E_1^{(j)} & 1 & 0 & \cdots & 0 \\
E_2^{(j+1)} & E_1^{(j+1)} & 1 & \ddots & \vdots \\
\vdots & \vdots &\ddots & \ddots & 0  \\
E_{i-j-k}^{(i-k-1)} & E_{i-j-k-1}^{(i-k-1)} & \cdots & E_1^{(i-k-1)} & 1 \\
(i-j-k)E_{i-j-k}^{(i-k-1)} & (i-j-k-1)E_{i-j-k-1}^{(i-k-1)} & \cdots & E_1^{(i-k-1)} & 0  
\end{array}
\right| = (-1)^{k+1} F_{i-k-1,j}
\end{align*}
where we used the cofactor expansion along the last column and Corollary~\ref{corollary:determinant_formula_for_Fij} for the equality above.
We proved \eqref{eq_proof_quantized_recursive_formula_for_Fij_Ck}. 

\noindent
\underline{\textbf{Proof of \eqref{eq_proof_quantized_recursive_formula_for_Fij_A}:}}
For each $\ell$-th row of $A$ except for the last row, we subtract the sum of $k$-th raw multiplied by $q_{j+k \ j+\ell-1}$ for $1 \leq k \leq \ell-1$ (regarded as $q_{j+\ell-1 \, j+\ell-1} = x_{j+\ell-1}$) from the $\ell$-th raw.
We proceed in steps from $\ell=i-j$ to $\ell=2$. 
By using the recursive formula \eqref{eq:recursive quantized elementary symmetric polynomials}, the result is 
\begin{align*}
\left|
\begin{array}{@{\,}cccccc@{\,}}
E_1^{(j)} & 1 & 0 & 0 & \cdots & 0 \\
E_2^{(j)}+q_{j \, j+1} & E_1^{(j)} & 1 & 0 & \ddots & \vdots \\
E_3^{(j+1)}+q_{j \, j+2} & E_2^{(j+1)} & E_1^{(j+1)} & 1 & \ddots & \vdots \\
\vdots & \vdots & \vdots & \ddots & \ddots & 0 \\
E_{i-j}^{(i-2)}+q_{j \, i-1} & E_{i-j-1}^{(i-2)} & E_{i-j-2}^{(i-2)} & \cdots & E_1^{(i-2)} & 1 \\
(i-j+1)(E_{i-j+1}^{(i-1)}+q_{ji}) & (i-j)E_{i-j}^{(i-1)} & (i-j-1)E_{i-j-1}^{(i-1)} & \cdots & 2E_2^{(i-1)} & E_1^{(i-1)}
\end{array}
\right|.
\end{align*}
Noting that $E_1^{(j)}= E_1^{(j-1)}+x_j$ for the $(1,1)$-entry, the determinant above is the sum of the following two types of determinants:
\begin{align*}
A_1 &\coloneqq \left|
\begin{array}{@{\,}cccccc@{\,}}
E_1^{(j-1)} & 1 & 0 & 0 & \cdots & 0 \\
E_2^{(j)} & E_1^{(j)} & 1 & 0 & \ddots & \vdots \\
E_3^{(j+1)} & E_2^{(j+1)} & E_1^{(j+1)} & 1 & \ddots & \vdots \\
\vdots & \vdots & \vdots & \ddots & \ddots & 0 \\
E_{i-j}^{(i-2)} & E_{i-j-1}^{(i-2)} & E_{i-j-2}^{(i-2)} & \cdots & E_1^{(i-2)} & 1 \\
(i-j+1)E_{i-j+1}^{(i-1)} & (i-j)E_{i-j}^{(i-1)} & (i-j-1)E_{i-j-1}^{(i-1)} & \cdots & 2E_2^{(i-1)} & E_1^{(i-1)}
\end{array}
\right|; \\
A_2 &\coloneqq \left|
\begin{array}{@{\,}cccccc@{\,}}
x_j & 1 & 0 & 0 & \cdots & 0 \\
q_{j \, j+1} & E_1^{(j)} & 1 & 0 & \ddots & \vdots \\
q_{j \, j+2} & E_2^{(j+1)} & E_1^{(j+1)} & 1 & \ddots & \vdots \\
\vdots & \vdots & \vdots & \ddots & \ddots & 0 \\
q_{j \, i-1} & E_{i-j-1}^{(i-2)} & E_{i-j-2}^{(i-2)} & \cdots & E_1^{(i-2)} & 1 \\
(i-j+1)q_{ji} & (i-j)E_{i-j}^{(i-1)} & (i-j-1)E_{i-j-1}^{(i-1)} & \cdots & 2E_2^{(i-1)} & E_1^{(i-1)}
\end{array}
\right|.
\end{align*}
In other words, we now have 
\begin{align*}
A = A_1 + A_2.
\end{align*}
By Corollary~\ref{corollary:determinant_formula_for_Fij} we have 
\begin{align*}
A_1 = F_{i-1,j-1}.
\end{align*}
It follows from the cofactor expansion along the first column for $A_2$, we also have 
\begin{align*}
A_2 = x_j \, F_{i-1,j} + \sum_{k=1}^{i-j-1} (-1)^k q_{j \, j+k} \, F_{i-1,j+k} +(-1)^{i-j}(i-j+1)q_{ji}
\end{align*}
by using Corollary~\ref{corollary:determinant_formula_for_Fij} again.
We proved \eqref{eq_proof_quantized_recursive_formula_for_Fij_A}. 

By \eqref{eq_proof_quantized_recursive_formula_for_Fij_1}, \eqref{eq_proof_quantized_recursive_formula_for_Fij_B}, 
\eqref{eq_proof_quantized_recursive_formula_for_Fij_Ck}, and \eqref{eq_proof_quantized_recursive_formula_for_Fij_A}, 
we obtain the desired recursive formula.
\end{proof}

\begin{example}
One can see from Theorem~\ref{theorem:quantized recursive formula for Fij} that
\begin{align*}
F_{j+1,j} = f_{j+1,j} -\sum_{k=1}^{j} 2q_{k \, k+1} = \sum_{k=1}^{j} \big( (x_k-x_{j+1})x_k -2q_{k \, k+1} \big) \ \ \ \textrm{for} \ j \geq 1. 
\end{align*}
In fact, it follows from Theorem~\ref{theorem:quantized recursive formula for Fij} that 
\begin{align*}
F_{k+1,k} - F_{k,k-1} = (x_k-x_{k+1})(x_1+\cdots+x_k) - 2q_{k \, k+1}
\end{align*}
for any $k \geq 1$. Thus, we have
\begin{align*}
F_{j+1,j}=\sum_{k=1}^j (F_{k+1,k} - F_{k,k-1}) = \sum_{k=1}^{j} \big( (x_k-x_{j+1})x_k -2q_{k \, k+1} \big). 
\end{align*}
\end{example}


\bigskip

\begin{thebibliography}{99}
\bibitem{ADGH}
H. Abe, L. DeDieu, F. Galetto, and M. Harada, 
\emph{Geometry of Hessenberg varieties with applications to Newton–Okounkov bodies}, 
Selecta Math. (N.S.) \textbf{24} (2018), no. 3, 2129--2163.

\bibitem{AbeInsko}
H. Abe and E. Insko, 
\emph{On singularity and normality of regular nilpotent Hessenberg varieties},
J. Algebra \textbf{651} (2024), 70--110.

\bibitem{AHHM}
H. Abe, M. Harada, T. Horiguchi, and M. Masuda,
\emph{The cohomology rings of regular nilpotent Hessenberg varieties in Lie type A}, 
Int. Math. Res. Not. IMRN \textbf{2019} (2019), 5316--5388. 

\bibitem{AHMMS}
T. Abe, T. Horiguchi, M. Masuda, S. Murai, and T. Sato,
\emph{Hessenberg varieties and hyperplane arrangements},
J. Reine Angew. Math. \textbf{764} (2020), 241--286.

\bibitem{AndTym}
D. Anderson and J. Tymoczko, 
\emph{Schubert polynomials and classes of Hessenberg varieties}, 
J. Algebra \textbf{323} (2010), no. 10, 2605--2623.

\bibitem{Bor53}
A. Borel, 
\emph{Sur la cohomologie des espaces fibr\'es principaux et des espaces homog\`enes de groupes de Lie compacts}, 
Ann. of Math. (2) \textbf{57} (1953), 115--207.

\bibitem{BroCho}
P. Brosnan and T. Y. Chow, 
\emph{Unit interval orders and the dot action on the cohomology of regular semisimple Hessenberg varieties}, 
Adv. Math. \textbf{329} (2018), 955--1001.

\bibitem{dMPS}
F. De Mari, C. Procesi, and M. A. Shayman,
\emph{Hessenberg varieties}, 
Trans. Amer. Math. Soc. {\bf 332} (1992),  no. 2, 529--534. 

\bibitem{dMS}
F. De Mari and M. A. Shayman, 
\emph{Generalized Eulerian numbers and the topology of the Hessenberg variety of a matrix}, 
Acta Appl. Math. \textbf{12} (1988), no. 3, 213--235.

\bibitem{FGP}
S. Fomin, S. Gelfand, and A. Postnikov, 
\emph{Quantum Schubert polynomials}, 
J. Amer. Math. Soc. \textbf{10} (1997), no. 3, 565--596.

\bibitem{Font95}
I. Ciocan-Fontanine, 
\emph{Quantum cohomology of flag varieties},
Int. Math. Res. Not. IMRN \textbf{1995} (1995), no. 6, 263--277.

\bibitem{Ful97}
W. Fulton,
\emph{Young Tableaux},
With applications to representation theory and geometry. London Mathematical Society Student Texts, \textbf{35}. Cambridge University Press, Cambridge, 1997. 

\bibitem{GK}
A. Givental and B. Kim, 
\emph{Quantum cohomology of flag manifolds and Toda lattices}, 
Comm. Math. Phys. \textbf{168} (1995), no. 3, 609--641.

\bibitem{Hor18}
T. Horiguchi, 
\emph{The cohomology rings of regular nilpotent Hessenberg varieties and Schubert polynomials},
Proc. Japan Acad., Ser. A \textbf{94}, no. 9 (2018), 87--92.

\bibitem{HorShi}
T. Horiguchi and T. Shirato, 
\emph{Coordinate rings of regular nilpotent Hessenberg varieties in the open opposite Schubert cell}, 
Forum Math. Sigma \textbf{13} (2025), Paper No. e44, 42 p. 

\bibitem{Kos}
B. Kostant, 
\emph{Flag manifold quantum cohomology, the Toda lattice, and the representation with highest weight $\rho$}, 
Selecta Math. (N.S.) \textbf{2} (1996), 43--91. 

\bibitem{Pet}
D. Peterson, 
\emph{Quantum cohomology of $G/P$}, 
Lecture Course, M.I.T., Spring Term, 1997.

\bibitem{Rie}
K. Rietsch, 
\emph{Totally positive Toeplitz matrices and quantum cohomology of partial flag varieties}, 
J. Amer. Math. Soc. \textbf{16} (2003), 363--392.

\bibitem{ShaWac}
J. Shareshian and M. L. Wachs, 
\emph{Chromatic quasisymmetric functions},
Adv. Math. \textbf{295} (2016), 497--551.

\bibitem{SomTym}
E. Sommers and J. Tymoczko, 
\emph{Exponents for $B$-stable ideals}, 
Trans. Amer. Math. Soc. \textbf{358} (2006), no. 8, 3493--3509.

\bibitem{Tym08}
J. Tymoczko, 
\emph{Permutation actions on equivariant cohomology of flag varieties},
Toric topology, Contemp. Math., \textbf{460} American Mathematical Society, Providence, RI, (2008), 365--384.
\end{thebibliography}
\end{document}